\documentclass[a4paper,oneside,reqno]{amsart}

\usepackage[top=30mm,bottom=30mm,left=30mm,right=30mm]{geometry}
\usepackage{amsmath,mathtools,amssymb,amsthm}
\usepackage[foot]{amsaddr}
\usepackage[shortlabels]{enumitem}
\usepackage[all]{xy}

\usepackage[
 bookmarksnumbered=true,
 colorlinks=true,
 allcolors=blue
]{hyperref}

\usepackage[
 backend=biber,
 style=alphabetic,
 sorting=nyvt
]{biblatex}

\newcommand{\closed}{\mathrm{cl}}
\newcommand{\etale}{\mathrm{\acute{e}t}}
\newcommand{\tame}{\mathrm{t}}
\newcommand{\abelian}{\mathrm{ab}}
\newcommand{\metabelian}{\mathrm{meta}}
\newcommand{\separable}{\mathrm{sep}}
\newcommand{\geometric}{\mathrm{geo}}
\newcommand{\cuspidal}{\mathrm{c}}
\newcommand{\noncuspidal}{\mathrm{nc}}
\newcommand{\Frobenius}{\mathrm{Fr}}
\newcommand{\FunctionField}[1]{K(#1)}

\DeclareMathOperator{\Aut}{Aut}
\DeclareMathOperator{\Inn}{Inn}
\DeclareMathOperator{\Out}{Out}
\DeclareMathOperator{\Isom}{Isom}
\DeclareMathOperator{\Spec}{Spec}
\DeclareMathOperator{\rank}{rank}

\DeclareMathOperator{\ord}{ord}
\let\plim\varprojlim

\newcommand{\GeneSubgrp}[1]{\langle #1 \rangle}
\newcommand{\GeneSubgrpTop}[1]{\overline{\GeneSubgrp{#1}}}
\newcommand{\Commensurator}[2]{\operatorname{Com}_{#1}\!\left(#2\right)}
\newcommand{\FinStepSolvQuo}[2]{#1^{#2}}
\newcommand{\FinStepSolvQuoGeom}[2]{#1^{(#2)}}
\newcommand{\AbsGalGrp}[1]{G_{#1}}
\newcommand{\EtFundGrpWithPt}[2]{\pi_{1}^{\etale}(#1,#2)}
\newcommand{\EtFundGrpTameWithPt}[2]{\pi_{1}^{\mathrm{tame}}(#1,#2)}
\newcommand{\EtFundGrpGeom}[1]{\Delta_{#1}}
\newcommand{\EtFundGrpPi}[1]{\Pi_{#1}}
\newcommand{\UnivCov}[1]{\widetilde{#1}}
\newcommand{\FieldAlgeClosure}[1]{\overline{#1}}
\newcommand{\FieldSepClosure}[1]{#1^{\separable}}
\newcommand{\CardinalityOfSet}[1]{\# #1}
\newcommand{\CenterSubgrp}[1]{\operatorname{Z}\!\left(#1\right)}

\newcommand{\Compactification}[1]{#1^{\mathrm{cpt}}}

\numberwithin{equation}{section}

\theoremstyle{plain}
\newtheorem{theorem}{Theorem}[section]
\newtheorem{Itheorem}{Theorem}

\newtheorem{proposition}[theorem]{Proposition}
\newtheorem{lemma}[theorem]{Lemma}
\newtheorem*{lemma*}{Lemma}
\newtheorem{corollary}[theorem]{Corollary}
\newtheorem*{conjecture*}{Conjecture}

\theoremstyle{definition}
\newtheorem{definition}[theorem]{Definition}
\newtheorem{notation}[theorem]{Notation}
\newtheorem{Inotation}[Itheorem]{Notation}
\newtheorem{remark}[theorem]{Remark}

\title[Metabelian Grothendieck conjecture for genus zero curves]{The metabelian Grothendieck conjecture for genus zero curves over finitely generated fields}
\date{Version of \today}
\author[N.~Yamaguchi]{Naganori Yamaguchi}
\address{Institute of Science Tokyo, 2-12-1 Ookayama, Meguro-ku, Tokyo 152-8550, Japan}
\email{yamaguchi.n.ac@m.titech.ac.jp}
\thanks{This work was supported by JSPS KAKENHI Grant Number 23KJ0881.}

\begin{document}

\begin{abstract}
	In this paper, we prove the metabelian Grothendieck conjecture for genus-zero curves over finitely generated fields.
	More precisely, we show that two hyperbolic genus-zero curves are isomorphic over the base field, up to Frobenius twist in positive characteristic, if and only if their geometrically maximal metabelian tame fundamental groups are isomorphic over the absolute Galois group of the base field.
\end{abstract}

\maketitle
\tableofcontents

\section*{Introduction}

In a letter to G.~Faltings, A.~Grothendieck proposed the following conjecture (see~\cite{MR1483108}):

\medskip

\begin{quote}\itshape
	Let $k$ be a field finitely generated over $\mathbb{Q}$, let $\FieldSepClosure{k}$ be a separable closure of $k$, and write $\AbsGalGrp{k}\coloneq \mathrm{Gal}(\FieldSepClosure{k}/k)$.
	Then the geometry of an ``\emph{anabelian}'' variety $X$ over $k$ may be reconstructed group-theoretically from the \'etale fundamental group $\EtFundGrpWithPt{X}{\ast}$ together with its natural projection $\mathrm{pr}\colon \EtFundGrpWithPt{X}{\ast}\to \AbsGalGrp{k}$.
\end{quote}

\medskip

In that letter, A.~Grothendieck did not specify precisely which varieties should be called anabelian.
He did, however, conjecture that hyperbolic curves are anabelian.
Here, we say that a smooth curve of type $(g,r)$ is \emph{hyperbolic} if $2-2g-r<0$.
We always assume that smooth curves are geometrically connected.
The Grothendieck conjecture for hyperbolic curves was proved in the genus-zero case by H.~Nakamura \cite{MR1072981}, in the affine case by A.~Tamagawa \cite{MR1478817}, and, in characteristic zero and in a form stronger than the original conjecture, by S.~Mochizuki \cite{MR1720187}; see also~\cite{MR1432110}.
For the positive-characteristic case over finitely generated fields, see J.~Stix \cite{MR2012864}.
The following consequence of Mochizuki's theorem is usually referred to as the \emph{weak form} of the Grothendieck conjecture in characteristic zero:

\medskip

\begin{quote}\itshape
	Let $k$ be a field finitely generated over $\mathbb{Q}$, and let $X_{1}$ and $X_{2}$ be hyperbolic curves over $k$.
	Then
	\begin{equation*}
		X_{1}
		\cong_{k}
		X_{2}
		\iff
		\EtFundGrpWithPt{X_{1}}{\ast}
		\cong_{\AbsGalGrp{k}}
		\EtFundGrpWithPt{X_{2}}{\ast}
	\end{equation*}
	(see~\cite[Theorem~16.5]{MR1720187}).
\end{quote}

\bigskip

\noindent
{\bf The metabelian truncation}

\bigskip

It is natural to consider truncated versions of the Grothendieck conjecture, namely, variants in which the \'etale fundamental group is replaced by a suitable group-theoretic quotient.
For example, the pro-$p$ case was studied by S.~Mochizuki \cite{MR1720187}, and the $m$-step solvable case for hyperbolic curves was studied by H.~Nakamura \cite{MR1040998}, by S.~Mochizuki \cite{MR1720187}, and by the author \cite{MR4578639,MR4745885}.
Since the Galois action on the maximal abelian quotient of the geometric étale fundamental group of a projective line minus rational points is described entirely by cyclotomic characters, this abelian quotient cannot recover the isomorphism class of the curve.
The metabelian quotient is therefore the first $m$-step solvable quotient for which a positive anabelian statement may still be expected.
In this paper, we investigate this quotient in the genus-$0$ case.

\medskip

To state our main results, we fix the following notation:
For any profinite group $G$, we write $G' \coloneq \overline{[G,G]}$, $G'' \coloneq \overline{[G',G']}$, and define $G^{\abelian}\coloneq G/G'$, $G^{\metabelian}\coloneq G/G''$.
We call them the \emph{maximal abelian quotient} and the \emph{maximal metabelian quotient} of $G$, respectively.
For any smooth curve $X$ over a field $k$ with separable closure $\FieldSepClosure{k}$, we write
\begin{equation*}
	\EtFundGrpGeom{X}^{\tame}
	\coloneq
	\EtFundGrpTameWithPt{X_{\FieldSepClosure{k}}}{\ast},
	\qquad
	\EtFundGrpPi{X}^{\tame}
	\coloneq
	\EtFundGrpTameWithPt{X}{\ast}.
\end{equation*}
We then define two natural quotients of $\EtFundGrpPi{X}^{\tame}$ by
\begin{equation*}
	\FinStepSolvQuoGeom{(\EtFundGrpPi{X}^{\tame})}{\abelian}
	\coloneq
	\EtFundGrpPi{X}^{\tame}/(\EtFundGrpGeom{X}^{\tame})',
	\qquad
	\FinStepSolvQuoGeom{(\EtFundGrpPi{X}^{\tame})}{\metabelian}
	\coloneq
	\EtFundGrpPi{X}^{\tame}/(\EtFundGrpGeom{X}^{\tame})''.
\end{equation*}
Here, to treat the positive-characteristic case uniformly, we work with the tame fundamental group (see~\cite{MR316453}) rather than the \'etale fundamental group.
In characteristic zero, these two groups coincide.

In this notation, the metabelian Grothendieck conjecture can be formulated as follows:

\medskip

\begin{quote}\itshape
	Let $k$ be a field finitely generated over $\mathbb{Q}$, and let $X_{1}$ and $X_{2}$ be hyperbolic curves over $k$.
	Then
	\begin{equation}\label{GCmetaconjeq}
		X_{1}
		\cong_{k}
		X_{2}
		\iff
		\FinStepSolvQuoGeom{(\EtFundGrpPi{X_{1}}^{\tame})}{\metabelian}
		\cong_{\AbsGalGrp{k}}
		\FinStepSolvQuoGeom{(\EtFundGrpPi{X_{2}}^{\tame})}{\metabelian}.
	\end{equation}
\end{quote}
\medskip
Several cases of the $m$-step solvable Grothendieck conjecture are already known; for example, the conjecture was proved for $m\geq 5$ by S.~Mochizuki \cite[Theorem~A$'$]{MR1720187}.
To the best of the author's knowledge, however, the only earlier result for the metabelian Grothendieck conjecture over a field of characteristic $0$ was the following special case due to H.~Nakamura:

\medskip

\begin{quote}\itshape
	Let $k$ be an algebraic number field satisfying either of the following conditions:
	\begin{enumerate}[(i)]
		\item\label{gwae}
		      $k$ is a quadratic field different from $\mathbb{Q}(\sqrt{2})$;
		\item\label{ewfaE}
		      there exists a prime ideal $\mathfrak{p}$ of $\mathcal{O}_{k}$, unramified in $k/\mathbb{Q}$, such that
		      \begin{equation*}
			      \CardinalityOfSet{\mathcal{O}_{k}/\mathfrak{p}}=2.
		      \end{equation*}
	\end{enumerate}
	Let $\lambda_{1},\lambda_{2}\in k\setminus\{0,1\}$ and define
	\begin{equation*}
		X_{1}\coloneq\mathbb{P}^{1}_{k}\setminus\{0,1,\infty,\lambda_{1}\},
		\qquad
		X_{2}\coloneq\mathbb{P}^{1}_{k}\setminus\{0,1,\infty,\lambda_{2}\}.
	\end{equation*}
	Then the equivalence \eqref{GCmetaconjeq} holds.
	(See~\cite[Theorem~A]{MR1040998}.)
\end{quote}
\medskip
Our first main theorem shows that the metabelian Grothendieck conjecture holds for genus-zero curves with $r\geq 5$:

\bigskip

\begin{Itheorem}[Corollary~\ref{fingeneweak}\ref{fingeneweakzero}]\label{Ithm_zero}
	Let $k$ be a field finitely generated over $\mathbb{Q}$, and let $X_{1}$ and $X_{2}$ be smooth curves over $k$.
	Assume that $g_{1}=0$ and that $r_{1}\geq 5$.
	Then
	\begin{equation*}
		X_{1}
		\cong_{k}
		X_{2}
		\iff
		\FinStepSolvQuoGeom{(\EtFundGrpPi{X_{1}}^{\tame})}{\metabelian}
		\cong_{\AbsGalGrp{k}}
		\FinStepSolvQuoGeom{(\EtFundGrpPi{X_{2}}^{\tame})}{\metabelian}.
	\end{equation*}
\end{Itheorem}

\bigskip

We next turn to the positive-characteristic case.
Over a field of characteristic $p>0$, relative Frobenius twists must be taken into account, since they induce isomorphisms on tame fundamental groups.
Thus, the natural positive-characteristic analogue of the metabelian Grothendieck conjecture must be formulated up to Frobenius twist.

\medskip

The finite-field case was treated by the author in \cite[Theorem~2.16]{MR4745885}, under suitable assumptions.
In this paper, we first refine that result to the geometrically maximal pro-prime-to-$p$ setting.
Here, for any field $k$ of characteristic $p\geq 0$ and any smooth curve $X$ over $k$, we write
\begin{equation*}
	\EtFundGrpGeom{X}
	\coloneq
	\EtFundGrpTameWithPt{X_{\FieldSepClosure{k}}}{\ast}^{p'},
	\qquad
	\EtFundGrpPi{X}
	\coloneq
	\EtFundGrpPi{X}^{\tame}
	/
	\ker(\EtFundGrpPi{X}^{\tame}\to\EtFundGrpGeom{X}).
\end{equation*}
Moreover, we similarly define $\FinStepSolvQuoGeom{\EtFundGrpPi{X}}{\abelian}$ and $\FinStepSolvQuoGeom{\EtFundGrpPi{X}}{\metabelian}$ as above.
The finite-field result established in this paper is as follows:

\bigskip

\begin{Itheorem}[Corollary~\ref{finitrelative}]\label{Ithm_positive_fin}
	Let $k$ be a finite field, and let $X_{1}$ and $X_{2}$ be smooth curves over $k$.
	Assume that $g_{1}=0$ and that $r_{1}\geq 5$.
	Then
	\begin{equation*}
		\FinStepSolvQuoGeom{\EtFundGrpPi{X_{1}}}{\metabelian}
		\cong_{\AbsGalGrp{k}}
		\FinStepSolvQuoGeom{\EtFundGrpPi{X_{2}}}{\metabelian}
	\end{equation*}
	if and only if there exists an integer $n$ such that $X_{1}(n)\cong_{k}X_{2}$.
\end{Itheorem}

\bigskip

The preceding theorem leads to the main positive-characteristic result of this paper.
Indeed, the specialization argument reduces the study of smooth curves over fields finitely generated over a finite field to the finite-field case.
By combining Theorem~\ref{Ithm_positive_fin} with this specialization argument, we obtain the following theorem over fields finitely generated over a finite field:

\bigskip

\begin{Itheorem}[Corollary~\ref{fingeneweak}\ref{fingeneweakpositive}]\label{Ithm_positive_inf}
	Let $k$ be a field finitely generated over a finite field, and let $X_{1}$ and $X_{2}$ be non-isotrivial smooth curves over $k$.
	Assume that $g_{1}=0$ and that $r_{1}\geq 5$.
	Then
	\begin{equation*}
		\FinStepSolvQuoGeom{\EtFundGrpPi{X_{1}}}{\metabelian}
		\cong_{\AbsGalGrp{k}}
		\FinStepSolvQuoGeom{\EtFundGrpPi{X_{2}}}{\metabelian},
	\end{equation*}
	if and only if there exists a pair $(n_{1},n_{2})$ of non-negative integers such that $X_{1}(n_{1})\cong_{k}X_{2}(n_{2})$.
\end{Itheorem}

\bigskip

The assumption $r_{1}\geq 5$ in Theorems~\ref{Ithm_zero},~\ref{Ithm_positive_fin}, and~\ref{Ithm_positive_inf} arises from our reconstruction of decomposition groups from the metabelian quotient.
More precisely, the argument requires a group-theoretical characterization to determine whether a given geometric quasi-section is cuspidal or non-cuspidal.
If such a characterization were established in the case $r_{1}=3$, then the same method would also apply to the case $r_{1}=4$.
At the time of writing, however, the author does not know how to obtain such a characterization in the case $r_{1}=3$, nor whether the metabelian Grothendieck conjecture holds in general when $r_{1}=4$; see Remark~\ref{r=5remark} for further discussion.
The case $r_{1}=3$, on the other hand, is exceptional, and the metabelian Grothendieck conjecture can be proved in this case; see Proposition~\ref{fingeneweak03}.

\medskip

The paper is organized as follows.
In Section~\ref{sectgtrvi}, we establish group-theoretical reconstructions of various invariants associated with the geometrically maximal abelian quotient.
In Section~\ref{sepsect}, we introduce geometric sections and quasi-sections, and prove separatedness properties for decomposition groups.
In Section~\ref{subsectionreco}, we reconstruct decomposition groups over finite fields from the metabelian quotient.
In Section~\ref{subsectionfinweak}, we prove Theorem~\ref{Ithm_positive_fin}.
Finally, in Section~\ref{subsectionfingeneweak}, we prove Theorems~\ref{Ithm_zero} and~\ref{Ithm_positive_inf}.

\section*{Notation and conventions}\label{sect1}

From now on, except in Theorems~\ref{finGCweak},~\ref{fingenemainthmchzero}, and~\ref{fingenemainthmchpositive}, we fix the following notation:

\begin{Inotation}\label{notation_basic1}
	Let $i=1,2$.
	Let $k_{i}$ be a field finitely generated over its prime field and of characteristic $p_{i}\geq 0$.
	Let $(\Compactification{X}_{i},E_{i})$ be a smooth curve of type $(0,r_{i})$ over $k_{i}$, and set $X_{i}\coloneq \Compactification{X}_{i}\setminus E_{i}$.
	If there is no risk of confusion, we write
	\begin{equation*}
		\EtFundGrpGeom{i},\qquad
		\FinStepSolvQuoGeom{\EtFundGrpPi{i}}{\abelian},\qquad
		\FinStepSolvQuoGeom{\EtFundGrpPi{i}}{\metabelian}
	\end{equation*}
	in place of
	$\EtFundGrpGeom{X_{i}}$, $\FinStepSolvQuoGeom{\EtFundGrpPi{X_{i}}}{\abelian}$, and $\FinStepSolvQuoGeom{\EtFundGrpPi{X_{i}}}{\metabelian}$,
	respectively.
	If there is no risk of confusion, we omit the index $i$ from this notation.
\end{Inotation}

Although several local arguments in this paper apply more generally to curves of $p$-rank $0$, the main theorems are proved only for genus-$0$ curves.
Accordingly, we do not treat the case of $p$-rank $0$ in full generality in this paper.
For notational simplicity, we also assume that $g_{1}=g_{2}=0$.

\section{Group-theoretical reconstructions of various invariants}\label{sectgtrvi}

In this section, we give group-theoretical reconstructions of various geometric invariants of smooth curves.
We first introduce the weight filtration of $\FinStepSolvQuo{\EtFundGrpGeom{}}{\abelian}$.
By construction, we have the following exact sequence, called the \emph{homotopy exact sequence}:
\begin{equation}\label{homexseq0}
	1
	\to\FinStepSolvQuo{\EtFundGrpGeom{}}{\abelian}
	\to\FinStepSolvQuoGeom{\EtFundGrpPi{}}{\abelian}
	\to\AbsGalGrp{k}
	\to 1
\end{equation}
This induces a Galois representation $\AbsGalGrp{k}\to \Aut(\FinStepSolvQuo{\EtFundGrpGeom{}}{\abelian})$.
We regard $\FinStepSolvQuo{\EtFundGrpGeom{}}{\abelian}$ as a $\AbsGalGrp{k}$-module via this Galois representation.

Let $\FinStepSolvQuoGeom{\UnivCov{\Compactification{X}}}{\abelian}$ denote the maximal geometrically abelian pro-prime-to-$p$ Galois covering of $\Compactification{X}$ that is \'etale over $X$, and for any closed point $\UnivCov{v}$ of $\FinStepSolvQuoGeom{\UnivCov{\Compactification{X}}}{\abelian}$, we define
\begin{equation*}
	D_{\UnivCov{v}}\coloneq D_{\UnivCov{v},\FinStepSolvQuoGeom{\EtFundGrpPi{}}{\abelian}}
	\coloneq
	\{\gamma\in \FinStepSolvQuoGeom{\EtFundGrpPi{}}{\abelian}\mid \gamma(\UnivCov{v})=\UnivCov{v}\}
\end{equation*}
\begin{equation*}
	I_{\UnivCov{v}}\coloneq I_{\UnivCov{v}, \FinStepSolvQuo{\EtFundGrpGeom{}}{\abelian}}
	\coloneq
	I_{\UnivCov{v}, \FinStepSolvQuoGeom{\EtFundGrpPi{}}{\abelian}}\coloneq \{\gamma\in D_{\UnivCov{v}}\mid \gamma \text{ acts trivially on the residue field }\kappa(\UnivCov{v})\},
\end{equation*}
and call these the \emph{decomposition group} at $\UnivCov{v}$ and the \emph{inertia group} at $\UnivCov{v}$, respectively.
In the abelian case, the inertia group $I_{\UnivCov{v}, \FinStepSolvQuo{\EtFundGrpGeom{}}{\abelian}}$ depends only on the image $v$ of $\UnivCov{v}$ under the natural projection $\FinStepSolvQuoGeom{\UnivCov{\Compactification{X}}}{\abelian}\to \Compactification{X}$.
Hence we sometimes write $I_{v, \FinStepSolvQuo{\EtFundGrpGeom{}}{\abelian}}$ instead of $I_{\UnivCov{v}, \FinStepSolvQuo{\EtFundGrpGeom{}}{\abelian}}$ for simplicity.
We write
\begin{equation*}
	\mathrm{Dec}(\FinStepSolvQuoGeom{\EtFundGrpPi{}}{\abelian}),\qquad
	\mathrm{Dec}^{\cuspidal}(\FinStepSolvQuoGeom{\EtFundGrpPi{}}{\abelian}),\qquad
	\mathrm{Dec}^{\noncuspidal}(\FinStepSolvQuoGeom{\EtFundGrpPi{}}{\abelian}),\qquad
	\mathrm{Iner}(\FinStepSolvQuoGeom{\EtFundGrpPi{}}{\abelian}),\qquad
	\mathrm{Iner}^{\neq\{1\}}(\FinStepSolvQuoGeom{\EtFundGrpPi{}}{\abelian})
\end{equation*}
for the $\FinStepSolvQuoGeom{\EtFundGrpPi{}}{\abelian}$-sets of decomposition groups, cuspidal decomposition groups (that is, decomposition groups at closed points of $E$), non-cuspidal decomposition groups, inertia groups, and nontrivial inertia groups, respectively.

\begin{lemma}\label{geomreco1}
	We keep the notation of Notation~\ref{notation_basic1}.
	Let
	\begin{equation*}
		\Phi\colon\FinStepSolvQuoGeom{\EtFundGrpPi{1}}{\metabelian}\to\FinStepSolvQuoGeom{\EtFundGrpPi{2}}{\metabelian}
	\end{equation*}
	be an isomorphism.
	Then the following hold:
	\begin{enumerate}[(1)]
		\item \label{finitereco}
		      The field $k_{1}$ is finite if and only if $k_{2}$ is finite.
		\item\label{geompireco}
		      The morphism $\Phi$ maps $\FinStepSolvQuo{\EtFundGrpGeom{1}}{\metabelian}$ to $\FinStepSolvQuo{\EtFundGrpGeom{2}}{\metabelian}$.
		      In particular, $\Phi$ induces a unique isomorphism $\FinStepSolvQuoGeom{\Phi}{\abelian}\colon\FinStepSolvQuoGeom{\EtFundGrpPi{1}}{\abelian}\to\FinStepSolvQuoGeom{\EtFundGrpPi{2}}{\abelian}$ such that the following diagram is commutative:
		      \begin{equation*}
			      \xymatrix{
			      \FinStepSolvQuoGeom{\EtFundGrpPi{1}}{\metabelian}
			      \ar@{->>}[d]
			      \ar[r]^-{\Phi}
			      &
			      \FinStepSolvQuoGeom{\EtFundGrpPi{2}}{\metabelian}
			      \ar@{->>}[d]
			      \\
			      \FinStepSolvQuoGeom{\EtFundGrpPi{1}}{\abelian}
			      \ar[r]^-{\FinStepSolvQuoGeom{\Phi}{\abelian}}
			      &
			      \FinStepSolvQuoGeom{\EtFundGrpPi{2}}{\abelian}
			      }
		      \end{equation*}
		      Here, the vertical arrows are the natural projections.
	\end{enumerate}
\end{lemma}

\begin{proof}
	\noindent\ref{finitereco}
	We have a surjection from a finitely generated free profinite group to $\EtFundGrpGeom{i}$ (see~\cite{MR0354651}), and hence $\FinStepSolvQuo{\EtFundGrpGeom{i}}{\metabelian}$ is also finitely generated.
	Moreover, $\AbsGalGrp{k_{i}}$ is finitely generated if and only if $k_{i}$ is finite.
	Thus, $k_{i}$ is finite if and only if $\FinStepSolvQuoGeom{\EtFundGrpPi{i}}{\metabelian}$ is finitely generated.
	In particular, the condition that $k_{i}$ is finite can be reconstructed group-theoretically from $\FinStepSolvQuoGeom{\EtFundGrpPi{i}}{\metabelian}$.
	This completes the proof.

	\noindent\ref{geompireco}
	When $k_{i}$ is finite, the group $\FinStepSolvQuo{\EtFundGrpGeom{i}}{\metabelian}$ coincides with the kernel of the natural morphism
	\begin{equation*}
		\FinStepSolvQuoGeom{\EtFundGrpPi{i}}{\metabelian}
		\twoheadrightarrow
		(\FinStepSolvQuoGeom{\EtFundGrpPi{i}}{\metabelian})^{\abelian}
		\big/
		\mathrm{tor}\left((\FinStepSolvQuoGeom{\EtFundGrpPi{i}}{\metabelian})^{\abelian}\right)
	\end{equation*}
	(see~\cite[Proposition~(3.3)(ii)]{MR1478817}).
	If $k_{i}$ is infinite, then $k_{i}$ is a Hilbertian field and $\AbsGalGrp{k_{i}}$ does not have a normal closed finitely generated subgroup (see~\cite[Theorem~13.4.2]{MR0868860} and \cite[Theorem~15.10]{MR2445111}).
	Hence $\FinStepSolvQuo{\EtFundGrpGeom{i}}{\metabelian}$ is the maximal normal closed finitely generated subgroup of $\FinStepSolvQuoGeom{\EtFundGrpPi{i}}{\metabelian}$.
	Thus, in both cases, $\FinStepSolvQuo{\EtFundGrpGeom{i}}{\metabelian}$, and hence $\FinStepSolvQuoGeom{\EtFundGrpPi{i}}{\abelian}$, can be reconstructed group-theoretically from $\FinStepSolvQuoGeom{\EtFundGrpPi{i}}{\metabelian}$.
	This proves the first sentence of \ref{geompireco}.
	The existence of $\FinStepSolvQuoGeom{\Phi}{\abelian}$ follows from this.
	The uniqueness of $\FinStepSolvQuoGeom{\Phi}{\abelian}$ follows from the surjectivity of the projection $\FinStepSolvQuoGeom{\EtFundGrpPi{i}}{\metabelian}\to\FinStepSolvQuoGeom{\EtFundGrpPi{i}}{\abelian}$.
	This completes the proof.
\end{proof}

Since $X$ has genus $0$, when $r\geq 1$ we have an exact sequence of $\AbsGalGrp{k}$-modules
\begin{equation}\label{wfseq}
	0
	\to \hat{\mathbb{Z}}^{p'}(1)
	\to \mathbb{Z}[E(\FieldSepClosure{k})]\bigotimes_{\mathbb{Z}} \hat{\mathbb{Z}}^{p'}(1)
	\xrightarrow{f}\FinStepSolvQuo{\EtFundGrpGeom{}}{\abelian}
	\to 0.
\end{equation}
Here, the group $\mathbb{Z}[E(\FieldSepClosure{k})]$ is the free $\mathbb{Z}$-module with basis $E(\FieldSepClosure{k})$ and is regarded as a $\AbsGalGrp{k}$-module via the natural $\AbsGalGrp{k}$-action on $E(\FieldSepClosure{k})$; the module $\hat{\mathbb{Z}}^{p'}(1)$ is the first Tate twist of $\hat{\mathbb{Z}}^{p'}$; the second arrow of the exact sequence \eqref{wfseq} is defined by sending $1$ to the diagonal $\sum_{v\in E(\FieldSepClosure{k})}v\otimes 1$; and $f$ is defined as the morphism sending $v\otimes 1$ to a (topological) generator of the inertia group $I_{v, \FinStepSolvQuo{\EtFundGrpGeom{}}{\abelian}}$ at $v\in E(\FieldSepClosure{k})$.

\begin{lemma}\label{geomreco2}
	We keep the notation of Notation~\ref{notation_basic1}.
	Assume that $\FinStepSolvQuo{\EtFundGrpGeom{}}{\abelian}$ is nontrivial; equivalently, assume that $r\geq 2$.
	Put $q\coloneq\CardinalityOfSet{k}$.
	Then the following hold:
	\begin{enumerate}[(1)]
		\item\label{rreco}
		      We have $r=\rank_{\hat{\mathbb{Z}}^{p'}}(\FinStepSolvQuo{\EtFundGrpGeom{}}{\abelian})+1$.
		\item\label{nodereco}
		      Define the $\AbsGalGrp{k}$-module
		      \begin{equation*}
			      W\coloneq \left(\bigwedge^{\mathrm{top}}\FinStepSolvQuo{\EtFundGrpGeom{}}{\abelian}\right)^{\otimes 2}
		      \end{equation*}
		      where $\bigwedge^{\mathrm{top}}$ stands for the highest exterior power of free $\hat{\mathbb{Z}}^{p'}$-modules and $\otimes$ denotes the tensor product of $\AbsGalGrp{k}$-modules.
		      Assume that $k$ is finite.
		      Then the coinvariants $\left(W\right)_{\AbsGalGrp{k}}$ are finite and we have
		      \begin{equation*}
			      q^{2(r-1)}
			      =
			      \CardinalityOfSet{\left(W\right)_{\AbsGalGrp{k}}}+1.
		      \end{equation*}
		\item\label{Frobreco}
		      Assume that $k$ is finite.
		      Let
		      \begin{equation*}
			      \rho\colon\AbsGalGrp{k}\to \Aut(W)=(\hat{\mathbb{Z}}^{p'})^{\times}
		      \end{equation*}
		      be the natural character.
		      Then the Frobenius element $\Frobenius_{k}\in \AbsGalGrp{k}$ is the unique element of $\AbsGalGrp{k}$ satisfying the following conditions:
		      \begin{enumerate}[(a)]
			      \item\label{Frobreco-a}
			            $\AbsGalGrp{k}$ is topologically generated by $\Frobenius_{k}$.
			      \item\label{Frobreco-b}
			            $\rho(\Frobenius_{k})\in p^{\mathbb{Z}_{\geq 0}}$.
		      \end{enumerate}
	\end{enumerate}
\end{lemma}

\begin{proof}
	\noindent\ref{rreco}
	The assertion follows from \eqref{wfseq}.

	\noindent\ref{nodereco}
	The $\AbsGalGrp{k}$-module
	\begin{equation*}
		\bigwedge^{\mathrm{top}}\FinStepSolvQuo{\EtFundGrpGeom{}}{\abelian}
	\end{equation*}
	has rank $1$ as a $\hat{\mathbb{Z}}^{p'}$-module with the $\AbsGalGrp{k}$-action
	\begin{equation}\label{rhocalu}
		\lambda^{\CardinalityOfSet{E(\FieldSepClosure{k})}-\CardinalityOfSet{E}}\chi^{r-1},
	\end{equation}
	where the morphism $\chi\colon\AbsGalGrp{k}\to (\hat{\mathbb{Z}}^{p'})^{\times}$ stands for the cyclotomic character and we set $\lambda\colon \AbsGalGrp{k}\cong\hat{\mathbb{Z}}\twoheadrightarrow\mathbb{Z}/2\mathbb{Z}\simeq \{\pm1\}\hookrightarrow (\hat{\mathbb{Z}}^{p'})^{\times}$.
	This implies that $W\cong \hat{\mathbb{Z}}^{p'}(2(r-1))$ as $\AbsGalGrp{k}$-modules.
	Hence the assertion follows from the fact that $\CardinalityOfSet{(\hat{\mathbb{Z}}^{p'}(1))_{\AbsGalGrp{k}}}=q-1$.

	\noindent\ref{Frobreco}
	By \eqref{rhocalu}, we have
	\begin{equation*}
		\rho(\Frobenius_{k})=q^{2(r-1)}.
	\end{equation*}
	In particular, $\rho(\Frobenius_{k})\in p^{\mathbb{Z}_{\geq 0}}$, and hence $\Frobenius_{k}$ satisfies conditions~\ref{Frobreco-a} and~\ref{Frobreco-b}.
	Conversely, let $\gamma\in \AbsGalGrp{k}$ be an element satisfying conditions~\ref{Frobreco-a} and~\ref{Frobreco-b}.
	Since $\AbsGalGrp{k}$ is topologically generated by $\Frobenius_{k}$, there exists $a\in \hat{\mathbb{Z}}^{\times}$ such that $\gamma=\Frobenius_{k}^{a}$.
	Hence
	\begin{equation*}
		\rho(\gamma)
		=
		\rho(\Frobenius_{k})^{a}
		=
		q^{2(r-1)a}
		=
		p^{2(r-1)fa},
	\end{equation*}
	where $f$ is the unique positive integer satisfying $q=p^{f}$.
	On the other hand, condition~\ref{Frobreco-b} implies that $\rho(\gamma)=\pm p^{n}$ for some $n\in \mathbb{Z}_{\geq 0}$, which implies that $2(r-1)fa=n$.
	Thus
	\begin{equation*}
		a
		=
		\frac{n}{2(r-1)f}
		\in
		\mathbb{Q}\cap \hat{\mathbb{Z}}^{\times}
		=
		\{\pm 1\}.
	\end{equation*}
	If $a=-1$, then $\rho(\gamma)=q^{-2(r-1)}\notin p^{\mathbb{Z}_{\geq 0}}$, which contradicts condition~\ref{Frobreco-b}.
	Therefore $a=1$; equivalently, $\gamma=\Frobenius_{k}$.
	This proves the uniqueness.
\end{proof}

\begin{lemma}\label{geomreco2_inerreco}
	We keep the notation of Notation~\ref{notation_basic1}.
	Let
	\begin{equation*}
		\Phi\colon
		\FinStepSolvQuoGeom{\EtFundGrpPi{1}}{\metabelian}
		\xrightarrow{\ \sim\ }
		\FinStepSolvQuoGeom{\EtFundGrpPi{2}}{\metabelian}
	\end{equation*}
	be an isomorphism.
	Assume that $\FinStepSolvQuo{\EtFundGrpGeom{1}}{\abelian}$ is nontrivial.
	Then the induced isomorphism $\FinStepSolvQuoGeom{\Phi}{\abelian}$ preserves inertia groups.
\end{lemma}

\begin{proof}
	The assertion follows from \cite[Proposition~1.12(1)]{MR4745885}.
\end{proof}

\section{Separatedness property for decomposition groups over finite fields}\label{sepsect}

In this section, we introduce sections and establish a separatedness property for decomposition groups and geometric quasi-sections using the Kummer homomorphism.
When $k$ is finite, a similar result is already known for decomposition groups of the maximal geometrically metabelian tame fundamental group (see~\cite[Proposition~2.6]{MR4745885}).
We begin with the following definition:

\begin{definition}\label{def:sections_m1}
	We keep the notation of Notation~\ref{notation_basic1}.
	Let $G$ be an open subgroup of $\AbsGalGrp{k}$, and let $\iota\coloneq \iota_{G}\colon G\hookrightarrow \AbsGalGrp{k}$ be the inclusion.
	We say that a continuous homomorphism
	\begin{equation*}
		s\colon G\to \FinStepSolvQuoGeom{\EtFundGrpPi{}}{\abelian}
	\end{equation*}
	is a \emph{section} of the restriction of $\mathrm{pr}$ if $\mathrm{pr}\circ s=\iota$.
	We denote the set of all sections by
	\begin{equation*}
		\mathcal{S}_{G}\coloneq
		\mathcal{S}_{G,\FinStepSolvQuoGeom{\EtFundGrpPi{}}{\abelian}}
	\end{equation*}
	A section $s\in \mathcal{S}_{G}$ is \emph{geometric} if there exists a closed point $\UnivCov{v}\in \FinStepSolvQuoGeom{\UnivCov{\Compactification{X}}}{\abelian}$ such that
	\begin{equation*}
		s(G)\subset D_{\UnivCov{v}}.
	\end{equation*}
	In this case, we say that $s$ is \emph{at} $\UnivCov{v}$ (or at its image $v\in \Compactification{X}$).
	When $v$ is a non-cuspidal point (resp.\ a cuspidal point), we say that $s$ is \emph{non-cuspidal} (resp.\ \emph{cuspidal}).
	We write
	\begin{equation*}
		\mathcal{S}^{\geometric}_{G},\qquad
		\mathcal{S}^{\geometric,\noncuspidal}_{G},\qquad
		\mathcal{S}^{\geometric,\cuspidal}_{G}
	\end{equation*}
	for the sets of geometric, non-cuspidal geometric, and cuspidal geometric sections, respectively.
\end{definition}

\begin{remark}\label{rem:sections_coh_m1}
	We keep the notation of Notation~\ref{notation_basic1}.
	Let $G$ be an open subgroup of $\AbsGalGrp{k}$.
	\begin{enumerate}[(1)]
		\item\label{rem:sections_coh_m1_1}
		      Let $u\in X(\FieldAlgeClosure{k}^{G})$ and let $\UnivCov{u}\in \FinStepSolvQuoGeom{\UnivCov{\Compactification{X}}}{\abelian}$ be a closed point lying over $u$.
		      The restriction $\mathrm{pr}|_{D_{\UnivCov{u}}}\colon D_{\UnivCov{u}}\twoheadrightarrow G$ is an isomorphism, and hence determines a section
		      \begin{equation*}
			      s_{\UnivCov{u}}\colon
			      G\xrightarrow{\ \sim\ }
			      D_{\UnivCov{u}}
			      \subset
			      \FinStepSolvQuoGeom{\EtFundGrpPi{}}{\abelian},
		      \end{equation*}
		      which is geometric by definition.
		      Conversely, if $s\colon G\to \FinStepSolvQuoGeom{\EtFundGrpPi{}}{\abelian}$ is a geometric section at a closed point $\UnivCov{v}\in \FinStepSolvQuoGeom{\UnivCov{\Compactification{X}}}{\abelian}$ lying over $v\in \Compactification{X}$, then $v$ is $\FieldAlgeClosure{k}^{G}$-rational (see~\cite[Lemma~2.6]{MR1478817}).
		\item\label{rem:sections_coh_m1_2}
		      Let $s_{0}\in \mathcal{S}_{G}$ and let $G$ act continuously on $\FinStepSolvQuo{\EtFundGrpGeom{}}{\abelian}$ by
		      \begin{equation*}
			      g\cdot \delta \coloneq s_{0}(g)\,\delta\,s_{0}(g)^{-1}
			      \qquad
			      (g\in G,\ \delta\in \FinStepSolvQuo{\EtFundGrpGeom{}}{\abelian}).
		      \end{equation*}
		      For $s\in \mathcal{S}_{G}$, define
		      \begin{equation*}
			      c_{s}(g)\coloneq s(g)\,s_{0}(g)^{-1}
			      \in
			      \FinStepSolvQuo{\EtFundGrpGeom{}}{\abelian}
			      \qquad
			      (g\in G).
		      \end{equation*}
		      Hence $c_{s}\in Z^{1}\left(G,\FinStepSolvQuo{\EtFundGrpGeom{}}{\abelian}\right)$, and the assignment $s\mapsto c_{s}$ induces a bijection
		      \begin{equation*}
			      \mathcal{S}_{G}
			      \xrightarrow{\ \sim\ }
			      Z^{1}_{\mathrm{cont}}\left(G,\FinStepSolvQuo{\EtFundGrpGeom{}}{\abelian}\right),
		      \end{equation*}
		      with inverse given by $c\mapsto (g\mapsto c(g)\,s_{0}(g))$.
		      Moreover, the action of $\FinStepSolvQuo{\EtFundGrpGeom{}}{\abelian}$ on $\mathcal{S}_{G}$ by conjugation corresponds to changing $c_{s}$ by a coboundary.
		      Hence we obtain a bijection
		      \begin{equation*}
			      \mathcal{S}_{G}/\FinStepSolvQuo{\EtFundGrpGeom{}}{\abelian}
			      \xrightarrow{\ \sim\ }
			      H^{1}_{\mathrm{cont}}\left(G,\FinStepSolvQuo{\EtFundGrpGeom{}}{\abelian}\right).
		      \end{equation*}
		      Similarly, for any $\UnivCov{x}\in \FinStepSolvQuoGeom{\UnivCov{\Compactification{X}}}{\abelian}$ and any open subgroup $H$ of $\mathrm{pr}(D_{\UnivCov{x}})$, we define
		      \begin{equation*}
			      \mathcal{S}_{H}(\UnivCov{x})
			      \coloneq
			      \left\{
			      s\colon H\to \FinStepSolvQuoGeom{\EtFundGrpPi{}}{\abelian}
			      \ \middle|\
			      \mathrm{pr}\circ s=\iota_{H},\ s(H)\subset D_{\UnivCov{x}}
			      \right\}.
		      \end{equation*}
		      By the same argument as above, we obtain a bijection
		      \begin{equation*}
			      \mathcal{S}_{H}(\UnivCov{x})/I_{\UnivCov{x}}
			      \xrightarrow{\ \sim\ }
			      H^{1}_{\mathrm{cont}}(H,I_{\UnivCov{x}}).
		      \end{equation*}
	\end{enumerate}
\end{remark}

\begin{lemma}\label{gsectionlema01}
	We keep the notation of Notation~\ref{notation_basic1}.
	Let $G$ be an open subgroup of $\AbsGalGrp{k}$.
	Assume that $k$ is perfect and that $X=\mathbb{G}_{m}(=\mathbb{P}^{1}_{k}-\{0,\infty\})$.
	Define
	\begin{equation*}
		j\coloneq j_{G,\FinStepSolvQuoGeom{\EtFundGrpPi{}}{\abelian}}\colon
		\mathcal{S}_{G}\times\mathcal{S}_{G}
		\to
		H^{1}_{\mathrm{cont}}\left(G,\FinStepSolvQuo{\EtFundGrpGeom{}}{\abelian}\right)
	\end{equation*}
	by sending $(s_{1},s_{2})$ to the cohomology class of the continuous $1$-cocycle
	\begin{equation*}
		G\to \FinStepSolvQuo{\EtFundGrpGeom{}}{\abelian},\qquad
		\sigma\mapsto s_{1}(\sigma)s_{2}(\sigma)^{-1}.
	\end{equation*}
	Let $s,s'\colon G\to \FinStepSolvQuoGeom{\EtFundGrpPi{}}{\abelian}$ be geometric sections at closed points $\UnivCov{v},\UnivCov{v}'$ of $\FinStepSolvQuoGeom{\UnivCov{\Compactification{X}}}{\abelian}$ lying over $v,v'\in \mathbb{G}_{m}(\FieldAlgeClosure{k}^{G})$, respectively.
	Then the image of $v/v'\in \mathbb{G}_{m}(\FieldAlgeClosure{k}^{G})$ under the Kummer homomorphism
	\begin{equation}\label{kummerg}
		\mathbb{G}_{m}(\FieldAlgeClosure{k}^{G}) \to \varprojlim_{p\nmid N} \mathbb{G}_{m}(\FieldAlgeClosure{k}^{G})/\mathbb{G}_{m}(\FieldAlgeClosure{k}^{G})^{N}\to H^{1}_{\mathrm{cont}}(G,T_{p'}(\mathbb{G}_{m}))
	\end{equation}
	coincides with $j(s,s')$.
\end{lemma}

\begin{proof}
	Let $N$ be a positive integer prime to $p$.
	Consider the finite \'etale Galois covering
	\begin{equation*}
		[N]:X
		=
		\mathbb{G}_{m}\to X
		=
		\mathbb{G}_{m},\qquad x\mapsto x^{N},
	\end{equation*}
	whose Galois group is $\mu_{N}$.
	Choose $N$-th roots $\sqrt[N]{v},\sqrt[N]{v'}\in \mathbb{G}_{m}(\FieldAlgeClosure{k})$ of $v,v'\in \mathbb{G}_{m}(\FieldAlgeClosure{k}^{G})$, respectively.
	By Remark~\ref{rem:sections_coh_m1}\ref{rem:sections_coh_m1_1}, the geometric sections $s,s'$ are induced by decomposition groups at suitable lifts $\UnivCov{v},\UnivCov{v}'$ lying over $v,v'$.
	Under the surjection $\FinStepSolvQuoGeom{\EtFundGrpPi{}}{\abelian}\twoheadrightarrow \mu_{N}$, the element $s(\sigma)\in \FinStepSolvQuoGeom{\EtFundGrpPi{}}{\abelian}$ acts on the fiber of $[N]$ over $v$; hence its image in $\mu_{N}$ is the unique element $\zeta_{\sigma}\in \mu_{N}$ such that
	\begin{equation*}
		\sigma(\sqrt[N]{v})
		=
		\zeta_{\sigma}\cdot \sqrt[N]{v}.
	\end{equation*}
	In other words, the image of $s(\sigma)$ in $\mu_{N}$ is
	\begin{equation*}
		\zeta_{\sigma}
		=
		\frac{\sigma(\sqrt[N]{v})}{\sqrt[N]{v}}\in \mu_{N}.
	\end{equation*}
	Similarly, the image of $s'(\sigma)$ in $\mu_{N}$ is $\sigma(\sqrt[N]{v'})/\sqrt[N]{v'}$.
	Therefore, the image in $\mu_{N}$ of the cocycle $\sigma\mapsto s(\sigma)s'(\sigma)^{-1}\in \FinStepSolvQuo{\EtFundGrpGeom{}}{\abelian}$ is the $\mu_{N}$-valued cocycle
	\begin{equation*}
		\sigma
		\mapsto
		\frac{\sigma(\sqrt[N]{v})}{\sqrt[N]{v}}\cdot
		\left(\frac{\sigma(\sqrt[N]{v'})}{\sqrt[N]{v'}}\right)^{-1}
		=
		\frac{\sigma(\sqrt[N]{v/v'})}{\sqrt[N]{v/v'}}.
	\end{equation*}
	This is exactly the cocycle representing the $N$-Kummer class of $v/v'$ in $H^{1}(G,\mu_{N})$.
	Hence, for every $N$ prime to $p$, the images of $j(s,s')$ and of $v/v'$ in $H^{1}(G,\mu_{N})$ coincide.
	Passing to the projective limit over such $N$, we obtain that the image of $v/v'$ under \eqref{kummerg} coincides with $j(s,s')$.
\end{proof}

\begin{lemma}\label{seppropr=2_lem}
	We keep the notation of Notation~\ref{notation_basic1}.
	Let $\UnivCov{v},\UnivCov{v}'\in \FinStepSolvQuoGeom{\UnivCov{\Compactification{X}}}{\abelian}$ be closed points lying over the same point $v\in \Compactification{X}_{\FieldAlgeClosure{k}}$.
	Assume that $k$ is perfect and that $r\geq 2$.
	If the image of $D_{\UnivCov{v}}\cap D_{\UnivCov{v}'}$ in $\AbsGalGrp{k}$ is open, then $\UnivCov{v}=\UnivCov{v}'$.
\end{lemma}

\begin{proof}
	Replacing $k$ by a finite extension, we may assume that $\mathrm{pr}(D_{\UnivCov{v}}\cap D_{\UnivCov{v}'})=\AbsGalGrp{k}$ and that $E(\FieldAlgeClosure{k})=E$.
	Hence $v$ is a $k$-rational point.
	We first claim that
	\begin{equation}\label{eq:H0-vanishing-sepprop}
		\left(\FinStepSolvQuo{\EtFundGrpGeom{}}{\abelian}/I_{v}
		\right)^{\AbsGalGrp{k}}
		=
		\{1\}.
	\end{equation}
	Indeed, since $E(\FieldAlgeClosure{k})=E$, the exact sequence \eqref{wfseq} implies that the $\ell$-primary part of $\FinStepSolvQuo{\EtFundGrpGeom{}}{\abelian}/I_{\UnivCov{v}'}$ is a quotient of a finite direct sum of copies of $\mathbb{Z}_{\ell}(1)$ for every prime $\ell\neq p$.
	On the other hand, since $k$ is finitely generated over the prime field, the group of $\ell$-power roots of unity in $k$ is finite, hence $(\mathbb{Z}_{\ell}(1))^{\AbsGalGrp{k}}=0$.
	This proves \eqref{eq:H0-vanishing-sepprop}.

	Since $\UnivCov{v}$ and $\UnivCov{v}'$ lie over the same point $v$, there exists $\gamma\in \FinStepSolvQuo{\EtFundGrpGeom{}}{\abelian}$ such that $\UnivCov{v}'=\gamma\UnivCov{v}$.
	Therefore $D_{\UnivCov{v}'}=\gamma D_{\UnivCov{v}}\gamma^{-1}$.
	Let $t\in \AbsGalGrp{k}$ and choose a lift $\tilde{t}\in D_{\UnivCov{v}}\cap D_{\UnivCov{v}'}$ of $t$.
	Since $\gamma\tilde{t}\gamma^{-1}\in D_{\UnivCov{v}'}$ and $\tilde{t}\in D_{\UnivCov{v}'}$, we have
	\begin{equation*}
		\tilde{t}^{-1}\gamma\tilde{t}\gamma^{-1}
		\in
		D_{\UnivCov{v}'}\cap \FinStepSolvQuo{\EtFundGrpGeom{}}{\abelian}
		=
		I_{\UnivCov{v}'}.
	\end{equation*}
	Thus the image $\overline{\gamma}\in \FinStepSolvQuo{\EtFundGrpGeom{}}{\abelian}/I_{\UnivCov{v}'}$ is fixed by $t$.
	Hence
	\begin{equation*}
		\overline{\gamma}
		\in
		\left(\FinStepSolvQuo{\EtFundGrpGeom{}}{\abelian}/I_{\UnivCov{v}'}\right)^{\AbsGalGrp{k}}.
	\end{equation*}
	Therefore $\overline{\gamma}=1$; equivalently, $\gamma\in I_{\UnivCov{v}'}$.
	As $I_{\UnivCov{v}'}$ fixes $\UnivCov{v}'$, we obtain $\UnivCov{v}=\gamma^{-1}\UnivCov{v}'=\UnivCov{v}'$.
\end{proof}

\begin{proposition}\label{sepprop}
	We keep the notation of Notation~\ref{notation_basic1}.
	Let $\UnivCov{v},\UnivCov{v}'\in\FinStepSolvQuoGeom{\UnivCov{\Compactification{X}}}{\abelian}$ be closed points.
	Assume that $k$ is perfect.
	Assume that one of the following conditions holds:
	\begin{enumerate}[(i)]
		\item
		      $r=2$ and $\{\UnivCov{v},\UnivCov{v}'\}\cap\FinStepSolvQuoGeom{\UnivCov{X}}{\abelian}\neq\emptyset$.
		\item
		      $r\geq 3$.
	\end{enumerate}
	Then the following conditions are equivalent:
	\begin{enumerate}[(a)]
		\item\label{sepprop-a}
		      $\UnivCov{v}=\UnivCov{v}'$;
		\item\label{sepprop-b}
		      $D_{\UnivCov{v}}=D_{\UnivCov{v}'}$;
		\item\label{sepprop-c}
		      $D_{\UnivCov{v}}$ and $D_{\UnivCov{v}'}$ are commensurable.
	\end{enumerate}
\end{proposition}

Note that the implication \ref{sepprop-c}$\Rightarrow$\ref{sepprop-a} does not hold when $r=2$ and $\UnivCov{v},\UnivCov{v}'\notin\FinStepSolvQuoGeom{\UnivCov{X}}{\abelian}$, since the decomposition groups $D_{\UnivCov{v}}$ and $D_{\UnivCov{v}'}$ are both open subgroups of $\FinStepSolvQuoGeom{\EtFundGrpPi{}}{\abelian}$.

\begin{proof}
	The implications \ref{sepprop-a}$\Rightarrow$\ref{sepprop-b}$\Rightarrow$\ref{sepprop-c} are immediate.
	We show \ref{sepprop-c}$\Rightarrow$\ref{sepprop-a}.
	Assume condition~\ref{sepprop-c}.

	We first consider the case $r=2$.
	When $\UnivCov{v}\neq\UnivCov{v}'$ and $\UnivCov{v}'\in\FinStepSolvQuoGeom{\UnivCov{E}}{\abelian}$, then $\UnivCov{v}\in\FinStepSolvQuoGeom{\UnivCov{X}}{\abelian}$ and the cuspidal decomposition group $D_{\UnivCov{v}'}$ is open in $\FinStepSolvQuoGeom{\EtFundGrpPi{}}{\abelian}$.
	On the other hand, the non-cuspidal decomposition group $D_{\UnivCov{v}}$ is not open in $\FinStepSolvQuoGeom{\EtFundGrpPi{}}{\abelian}$, since $D_{\UnivCov{v}}\cap \FinStepSolvQuoGeom{\EtFundGrpGeom{}}{\abelian}=1$.
	Therefore, the condition \ref{sepprop-c} does not hold.
	Hence we may assume that $\UnivCov{v},\UnivCov{v}'\in\FinStepSolvQuoGeom{\UnivCov{X}}{\abelian}$.
	Replacing $k$ by a finite extension, we may assume that $\mathrm{pr}(D_{\UnivCov{v}}\cap D_{\UnivCov{v}'})=\AbsGalGrp{k}$ and that $E(\FieldAlgeClosure{k})=E$.
	Let $v,v'\in \Compactification{X}_{\FieldAlgeClosure{k}}(k)$ be the images of $\UnivCov{v},\UnivCov{v}'$, respectively.
	After replacing $\Compactification{X}$ by an isomorphic curve, we may choose a coordinate on $\Compactification{X}\simeq \mathbb{P}^{1}_{k}$ such that $v,v'\in X\subset \mathbb{G}_{m}=\mathbb{P}^{1}_{k}-\{0,\infty\}$.
	Since $D_{\UnivCov{v}}\cap D_{\UnivCov{v}'}\twoheadrightarrow \AbsGalGrp{k}$, the two geometric sections $s_{\UnivCov{v}}$ and $s_{\UnivCov{v}'}$ coincide.
	Thus, the class $j(s_{\UnivCov{v}},s_{\UnivCov{v}'})$ is trivial.
	By Lemma~\ref{gsectionlema01}, this implies that the Kummer image of $v/v'\in \mathbb{G}_{m}(k)$ is trivial.
	By the injectivity of the Kummer homomorphism \eqref{kummerg}, we obtain $v/v'=1$; equivalently, $v=v'$.
	Since $D_{\UnivCov{v}}\cap D_{\UnivCov{v}'}$ has finite index in $D_{\UnivCov{v}}$, its image in $\AbsGalGrp{k}$ is open.
	Thus, by Lemma~\ref{seppropr=2_lem}, we obtain $\UnivCov{v}=\UnivCov{v}'$.
	This completes the proof of the case of $r=2$.

	Next, we consider the case $r\geq 3$.
	In this case, the inertia subgroups $I_{\UnivCov{v}}$ and $I_{\UnivCov{v}'}$ are also commensurable.
	Thus, if $\UnivCov{v}$ and $\UnivCov{v}'$ are both cuspidal, then the images of $\UnivCov{v}$ and $\UnivCov{v}'$ in $E_{\FieldAlgeClosure{k}}$ coincide by the separatedness property of inertia subgroups (see~\cite[Lemma~1.11(1)]{MR4745885}).
	Therefore, condition~\ref{sepprop-a} holds.
	Thus, we may assume that $\UnivCov{v}, \UnivCov{v}' \in \FinStepSolvQuoGeom{\UnivCov{X}}{\abelian}$.
	After replacing $\Compactification{X}$ by an isomorphic curve, we may choose a coordinate on $\Compactification{X}\simeq \mathbb{P}^{1}_{k}$ such that $X\subset \mathbb{G}_{m}=\mathbb{P}^{1}_{k}-\{0,\infty\}$ and $v\in \mathbb{G}_{m}$.
	Since the commensurability is preserved by the morphism
	\begin{equation*}
		\FinStepSolvQuoGeom{\EtFundGrpPi{X}}{\abelian}
		\to
		\FinStepSolvQuoGeom{\EtFundGrpPi{\mathbb{G}_{m}}}{\abelian}
	\end{equation*}
	induced by the open immersion, the case of $r=2$ implies $v=v'$.
	Since $D_{\UnivCov{v}}\cap D_{\UnivCov{v}'}$ has finite index in $D_{\UnivCov{v}}$, its image in $\AbsGalGrp{k}$ is open.
	Thus, by Lemma~\ref{seppropr=2_lem}, we obtain $\UnivCov{v}=\UnivCov{v}'$.
	This completes the proof.
\end{proof}

\begin{remark}\label{rem1.8}
	We keep the notation of Notation~\ref{notation_basic1}.
	Assume that $k$ is perfect and that $r\geq 3$ (resp.\ $r=2$).
	An important consequence of Proposition~\ref{sepprop}\ref{sepprop-a}$\Leftrightarrow$\ref{sepprop-c} is that the following natural map is bijective:
	\begin{equation*}
		(\FinStepSolvQuoGeom{\UnivCov{\Compactification{X}}}{\abelian})^{\closed}
		\to
		\mathrm{Dec}\left(\FinStepSolvQuoGeom{\EtFundGrpPi{}}{\abelian}\right)
		\qquad
		(
		\text{resp. }
			(\FinStepSolvQuoGeom{\UnivCov{X}}{\abelian})^{\closed}
		\to
		\mathrm{Dec}^{\noncuspidal}\left(\FinStepSolvQuoGeom{\EtFundGrpPi{}}{\abelian}\right)
		).
	\end{equation*}
	In particular, once $\mathrm{Dec}\left(\FinStepSolvQuoGeom{\EtFundGrpPi{}}{\abelian}\right)$ is reconstructed group-theoretically from $\FinStepSolvQuoGeom{\EtFundGrpPi{}}{\metabelian}$, this bijectivity allows us to reconstruct $\FinStepSolvQuoGeom{\UnivCov{\Compactification{X}}}{\abelian}$ itself group-theoretically from $\FinStepSolvQuoGeom{\EtFundGrpPi{}}{\metabelian}$.
	This strategy was used in \cite{MR1478817} in the classical setting (and in \cite{MR4745885} in the metabelian setting).
\end{remark}

\begin{definition}\label{def:norm_comm}
	Let $G$ be a topological group, and let $H$ be a closed subgroup of $G$.
	We define the \emph{commensurator} of $H$ in $G$ by
	\begin{equation*}
		\Commensurator{G}{H}\coloneq \{x\in G\mid xHx^{-1}\text{ and }H\text{ are commensurable}\}.
	\end{equation*}
\end{definition}

\begin{corollary}\label{cor2.6comm}
	We keep the notation of Notation~\ref{notation_basic1}.
	Let $\UnivCov{v}\in \FinStepSolvQuoGeom{\UnivCov{\Compactification{X}}}{\abelian}$ be a closed point.
	Assume that $k$ is perfect and that $r\geq 3$.
	Then for any open subgroup $H$ of $D_{\UnivCov{v}}$, we have
	\begin{equation*}
		D_{\UnivCov{v}}
		=
		\Commensurator{\FinStepSolvQuoGeom{\EtFundGrpPi{}}{\abelian}}{H}.
	\end{equation*}
\end{corollary}

\begin{proof}
	Since $H$ is an open subgroup of $D_{\UnivCov{v}}$, we have $D_{\UnivCov{v}}\subset \Commensurator{\FinStepSolvQuoGeom{\EtFundGrpPi{}}{\abelian}}{H}$.
	Let $c\in \Commensurator{\FinStepSolvQuoGeom{\EtFundGrpPi{}}{\abelian}}{H}$ be an element.
	Since $H$ and $cHc^{-1}$ are commensurable by hypothesis and $H$ is open in $D_{\UnivCov{v}}$, it follows that $D_{\UnivCov{v}}$ and $c\cdot D_{\UnivCov{v}}\cdot c^{-1}(=D_{c\UnivCov{v}})$ are commensurable.
	By Proposition~\ref{sepprop}\ref{sepprop-a}$\Leftrightarrow$\ref{sepprop-c}, this implies $c\cdot \UnivCov{v}=\UnivCov{v}$.
	Therefore $c\in D_{\UnivCov{v}}$, which proves the reverse inclusion.
\end{proof}

Next, we introduce quasi-sections and show a separatedness property for them.

\begin{definition}\label{def:quasi-sections}
	We keep the notation of Notation~\ref{notation_basic1}.
	We define the $\FinStepSolvQuoGeom{\EtFundGrpPi{}}{\abelian}$-set
	\begin{equation*}
		\mathcal{QS}
		\coloneq
		\varinjlim_{G\subset \AbsGalGrp{k}: \mathrm{open}}
		\mathcal{S}_{G},
	\end{equation*}
	where the transition maps are the restrictions $s\mapsto s|_{G'}$.
	We call an element of $\mathcal{QS}$ a \emph{quasi-section}.
	Thus, a quasi-section $\mathfrak{s}$ is represented by a pair $(G,s)$ with $G$ an open subgroup of $\AbsGalGrp{k}$ and $s\in \mathcal{S}_{G}$.
	If there is no risk of confusion, we write
	\begin{equation*}
		(G,s)
		\in
		\mathfrak{s}
	\end{equation*}
	for such a representative.
	A quasi-section $\mathfrak{s}$ is \emph{geometric} if there exists a representative $(G,s)$ such that $s$ is geometric.
	When a representative $s$ is at $\UnivCov{v}\in\FinStepSolvQuoGeom{\UnivCov{\Compactification{X}}}{\abelian}$, we say that $\mathfrak{s}$ is \emph{at} $\UnivCov{v}$ (or \emph{at} its image $v\in \Compactification{X}$).
	If such a representative $s$ is non-cuspidal (resp.\ cuspidal), then we say that $\mathfrak{s}$ is \emph{non-cuspidal} (resp.\ \emph{cuspidal}).
	We write
	\begin{equation*}
		\mathcal{QS}^{\geometric,\noncuspidal},
		\qquad
		\mathcal{QS}^{\geometric,\cuspidal},
		\qquad
		\mathcal{QS}^{\geometric}
	\end{equation*}
	for the sets of non-cuspidal geometric, cuspidal geometric, and geometric quasi-sections, respectively.
\end{definition}

\begin{lemma}\label{sepprop_e}
	We keep the notation of Notation~\ref{notation_basic1}.
	Let $\UnivCov{v},\UnivCov{v}'$ be closed points of $\FinStepSolvQuoGeom{\UnivCov{\Compactification{X}}}{\abelian}$.
	Assume that $k$ is perfect and that the following condition holds:
	\begin{equation*}
		r\geq \CardinalityOfSet{\left(\{\UnivCov{v},\UnivCov{v}'\}\cap \FinStepSolvQuoGeom{\UnivCov{E}}{\abelian}\right)}+2.
	\end{equation*}
	Then the conditions \ref{sepprop-a}--\ref{sepprop-c} in Proposition~\ref{sepprop} are equivalent to the following condition:
	\begin{enumerate}[(a), start=4]
		\item\label{sepprop-e} The image of $D_{\UnivCov{v}}\cap D_{\UnivCov{v}'}$ in $\AbsGalGrp{k}$ is open.
	\end{enumerate}
\end{lemma}

\begin{proof}
	The implication \ref{sepprop-c}$\Rightarrow$\ref{sepprop-e} is immediate.
	Replacing $k$ by a finite extension, we may assume that $\mathrm{pr}(D_{\UnivCov{v}}\cap D_{\UnivCov{v}'})=\AbsGalGrp{k}$ and that $E(\FieldAlgeClosure{k})=E$.
	Let $v,v'\in \Compactification{X}_{\FieldAlgeClosure{k}}$ be the images of $\UnivCov{v},\UnivCov{v}'$, respectively.
	Hence $v$ and $v'$ are $k$-rational.
	After replacing $\Compactification{X}$ by an isomorphic curve, we may choose a coordinate on $\Compactification{X}\simeq \mathbb{P}^{1}_{k}$ such that $X\subset \mathbb{G}_{m}=\mathbb{P}^{1}_{k}-\{0,\infty\}$ and $v,v'\in \mathbb{G}_{m}$.
	(Here, we use the hypotheses in the statement.)
	Let
	\begin{equation*}
		\FinStepSolvQuoGeom{\EtFundGrpPi{X}}{\abelian}\to
		\FinStepSolvQuoGeom{\EtFundGrpPi{\mathbb{G}_{m}}}{\abelian}
	\end{equation*}
	be the morphism induced by the open immersion.
	Push forward the geometric sections attached to $\UnivCov{v}$ and $\UnivCov{v}'$ to obtain geometric sections for $\mathbb{G}_{m}$ at $v,v'\in \mathbb{G}_{m}(k)$.
	Since $D_{\UnivCov{v}}\cap D_{\UnivCov{v}'}\twoheadrightarrow \AbsGalGrp{k}$, the two sections coincide on $\AbsGalGrp{k}$.
	Thus, the class $j(s_{\UnivCov{v}},s_{\UnivCov{v}'})$ is trivial.
	By Lemma~\ref{gsectionlema01}, this implies that the Kummer image of $v/v'\in \mathbb{G}_{m}(k)$ is trivial.
	By the injectivity of the Kummer homomorphism \eqref{kummerg}, we obtain $v/v'=1$; equivalently, $v=v'$.
	Thus, by Lemma~\ref{seppropr=2_lem}, we obtain $\UnivCov{v}=\UnivCov{v}'$.
	This completes the proof.
\end{proof}

\begin{proposition}\label{sepprop_quasisection}
	We keep the notation of Notation~\ref{notation_basic1}.
	Assume that $k$ is perfect.
	Then the following hold:
	\begin{enumerate}[(1)]
		\item \label{sepprop_quasisection_1}
		      Assume that $r\geq 2$.
		      Then every non-cuspidal geometric quasi-section $\mathfrak{s}\in \mathcal{QS}^{\geometric,\noncuspidal}$ is at a unique closed point $\UnivCov{x}_{\mathfrak{s}}\in (\FinStepSolvQuoGeom{\UnivCov{X}}{\abelian})^{\closed}$.
		      In particular, the map
		      \begin{equation}\label{qsection-nc-point-map}
			      \UnivCov{x}_{\ast}\colon
			      \mathcal{QS}^{\geometric,\noncuspidal}
			      \to
			      (\FinStepSolvQuoGeom{\UnivCov{X}}{\abelian})^{\closed},
			      \qquad
			      \mathfrak{s}\mapsto \UnivCov{x}_{\mathfrak{s}}
		      \end{equation}
		      is a bijection.
		\item \label{sepprop_quasisection_2}
		      Assume that $r\geq 3$.
		      Then
		      \begin{equation*}
			      \mathcal{QS}^{\geometric,\noncuspidal}
			      \cap
			      \mathcal{QS}^{\geometric,\cuspidal}
			      =
			      \emptyset.
		      \end{equation*}
		      In other words, for a geometric quasi-section, the property of being cuspidal or non-cuspidal is intrinsic.

		\item \label{sepprop_quasisection_3}
		      Assume that $r\geq 4$.
		      Then every geometric quasi-section $\mathfrak{s}\in \mathcal{QS}^{\geometric}$ is at a unique closed point $\UnivCov{x}_{\mathfrak{s}}\in (\FinStepSolvQuoGeom{\UnivCov{\Compactification{X}}}{\abelian})^{\closed}$.
		      In particular, the surjection
		      \begin{equation}\label{qsection-point-map}
			      \UnivCov{x}_{\ast}\colon
			      \mathcal{QS}^{\geometric}
			      \to
			      (\FinStepSolvQuoGeom{\UnivCov{\Compactification{X}}}{\abelian})^{\closed},
			      \qquad
			      \mathfrak{s}\mapsto \UnivCov{x}_{\mathfrak{s}}
		      \end{equation}
		      is well-defined and extends the bijection \eqref{qsection-nc-point-map}.
	\end{enumerate}
\end{proposition}

\begin{proof}
	\noindent\ref{sepprop_quasisection_1}
	Let $\mathfrak{s}\in \mathcal{QS}^{\geometric,\noncuspidal}$.
	Assume that $\mathfrak{s}$ is at closed points $\UnivCov{u},\UnivCov{u}'\in \FinStepSolvQuoGeom{\UnivCov{X}}{\abelian}$.
	Hence there exist representatives $(G,s), (G',s')\in\mathfrak{s}$ with $s(G)\subset D_{\UnivCov{u}}$ and $s'(G')\subset D_{\UnivCov{u}'}$.
	Since they represent the same quasi-section, after replacing them by their restrictions to a suitable open subgroup of $G\cap G'$, we may assume that $(G,s)=(G',s')$ and that $s(G)\subset D_{\UnivCov{u}}\cap D_{\UnivCov{u}'}$.
	Since $\mathrm{pr}(s(G))=G$ is open in $\AbsGalGrp{k}$, the image of $D_{\UnivCov{u}}\cap D_{\UnivCov{u}'}$ in $\AbsGalGrp{k}$ is open.
	Hence Lemma~\ref{sepprop_e} implies $\UnivCov{u}=\UnivCov{u}'$.
	(Here, we use the assumption $r\geq 2$.)
	This proves the uniqueness of $\UnivCov{x}_{\mathfrak{s}}$.
	Hence the map~\eqref{qsection-nc-point-map} is well-defined.
	Since $\mathrm{pr}$ is injective on each non-cuspidal decomposition group, the injectivity of the map \eqref{qsection-nc-point-map} follows.
	The surjectivity of the map \eqref{qsection-nc-point-map} follows from Remark~\ref{rem:sections_coh_m1}\ref{rem:sections_coh_m1_1}.
	This completes the proof.

	\noindent\ref{sepprop_quasisection_2}
	Assume that there exists $\mathfrak{s}\in \mathcal{QS}^{\geometric,\noncuspidal}\cap \mathcal{QS}^{\geometric,\cuspidal}$.
	Hence we may choose a representative $(G,s)\in\mathfrak{s}$, a closed point $\UnivCov{u}\in \FinStepSolvQuoGeom{\UnivCov{X}}{\abelian}$, and a closed point $\UnivCov{e}\in \FinStepSolvQuoGeom{\UnivCov{E}}{\abelian}$ such that $s(G)\subset D_{\UnivCov{u}}\cap D_{\UnivCov{e}}$.
	Since $\mathrm{pr}(s(G))=G$ is open in $\AbsGalGrp{k}$, the image of $D_{\UnivCov{u}}\cap D_{\UnivCov{e}}$ in $\AbsGalGrp{k}$ is open.
	Hence Lemma~\ref{sepprop_e} implies $\UnivCov{u}=\UnivCov{e}$, which is impossible.
	(Here, we use the assumption $r\geq 3$.)
	This completes the proof.

	\noindent\ref{sepprop_quasisection_3}
	Let $\mathfrak{s}\in \mathcal{QS}^{\geometric}$.
	Assume that $\mathfrak{s}$ is at closed points $\UnivCov{x},\UnivCov{x}'\in \FinStepSolvQuoGeom{\UnivCov{\Compactification{X}}}{\abelian}$.
	Hence there exist representatives $(G,s), (G',s')\in\mathfrak{s}$ with $s(G)\subset D_{\UnivCov{x}}$ and $s'(G')\subset D_{\UnivCov{x}'}$.
	Since they represent the same quasi-section, after replacing them by their restrictions to a suitable open subgroup of $G\cap G'$, we may assume that $(G,s)=(G',s')$ and that $s(G)\subset D_{\UnivCov{x}}\cap D_{\UnivCov{x}'}$.
	Since $\mathrm{pr}(s(G))=G$ is open in $\AbsGalGrp{k}$, the image of $D_{\UnivCov{x}}\cap D_{\UnivCov{x}'}$ in $\AbsGalGrp{k}$ is open.
	Hence Lemma~\ref{sepprop_e} implies $\UnivCov{x}=\UnivCov{x}'$.
	(Here, we use the assumption $r\geq 4$.)
	Thus $\UnivCov{x}_{\mathfrak{s}}$ is well-defined.
	The surjectivity of the map \eqref{qsection-point-map} follows from Remark~\ref{rem:sections_coh_m1}\ref{rem:sections_coh_m1_1}.
	This completes the proof.
	The last assertion is immediate from~\ref{sepprop_quasisection_1}.
\end{proof}

\begin{notation}
	Assume that $k$ is perfect and that $r\geq 4$.
	For any $\mathfrak{s}\in \mathcal{QS}^{\geometric}$, we define $D_{\mathfrak{s}}\coloneq D_{\UnivCov{x}_{\mathfrak{s}}}$, $I_{\mathfrak{s}}\coloneq I_{\UnivCov{x}_{\mathfrak{s}}}$, and let $x_{\mathfrak{s}}\in \Compactification{X}_{\FieldAlgeClosure{k}}$ be the image of $\UnivCov{x}_{\mathfrak{s}}$.
\end{notation}

\section{A group-theoretical reconstruction of decomposition groups over finite fields}\label{subsectionreco}

In this section, we reconstruct decomposition groups group-theoretically when $k$ is finite (see Proposition~\ref{bidecoabel} below).

\begin{definition}\label{def:two-cusp-quotient}
	We keep the notation of Notation~\ref{notation_basic1}.
	Let $\mathcal{I}$ be a two-element subset of $\mathrm{Iner}^{\neq\{1\}}(\FinStepSolvQuoGeom{\EtFundGrpPi{}}{\abelian})$.
	Define the closed normal subgroup $N_{\mathcal{I}}$ of $\FinStepSolvQuoGeom{\EtFundGrpPi{}}{\abelian}$ to be the subgroup topologically generated by all inertia groups other than those in $\mathcal{I}$.
	We write
	\begin{equation*}
		\rho_{\mathcal{I}}\colon
		\FinStepSolvQuoGeom{\EtFundGrpPi{}}{\abelian}
		\twoheadrightarrow
		\FinStepSolvQuoGeom{\EtFundGrpPi{}}{\abelian}\big/N_{\mathcal{I}}
	\end{equation*}
	for the natural projection.
	We also define a map
	\begin{equation*}
		\rho_{\mathcal{I}}^{\mathrm{qs}}\colon \mathcal{QS}
		\to
		\mathcal{QS}_{\FinStepSolvQuoGeom{\EtFundGrpPi{}}{\abelian}/N_{\mathcal{I}}}
	\end{equation*}
	by declaring that, if $(G,s)\in\mathfrak{s}$, then $\rho_{\mathcal{I}}^{\mathrm{qs}}(\mathfrak{s})$ is represented by $(G,\rho_{\mathcal{I}}\circ s)$.
\end{definition}

\begin{definition}\label{decorel}
	We keep the notation of Notation~\ref{notation_basic1}.
	Assume that $r\geq 3$.
	For $\mathfrak{s},\mathfrak{s}'\in \mathcal{QS}^{\geometric}$, we write $\mathfrak{s}\mathrm{R}\mathfrak{s}'$ if there exists $I\in \mathrm{Iner}^{\neq\{1\}}\left(\FinStepSolvQuoGeom{\EtFundGrpPi{}}{\abelian}\right)$ such that, for every two-element subset
	\begin{equation*}
		\mathcal{I}\subset
		\mathrm{Iner}^{\neq\{1\}}\left(\FinStepSolvQuoGeom{\EtFundGrpPi{}}{\abelian}\right)\setminus\{I\},
	\end{equation*}
	we have $\rho^{\mathrm{qs}}_{\mathcal{I}}(\mathfrak{s})=a\cdot\rho^{\mathrm{qs}}_{\mathcal{I}}(\mathfrak{s}')$ as elements of $\mathcal{QS}_{\FinStepSolvQuoGeom{\EtFundGrpPi{}}{\abelian}/N_{\mathcal{I}}}$ for some $a\in\FinStepSolvQuoGeom{\EtFundGrpPi{}}{\abelian}/N_{\mathcal{I}}$.
	The binary relation $\mathrm{R}$ is reflexive and symmetric by definition.
	We do not claim a priori that it is transitive.
\end{definition}

\begin{lemma}\label{decorellem}
	We keep the notation of Notation~\ref{notation_basic1}.
	Let $\mathfrak{s},\mathfrak{s}'\in \mathcal{QS}^{\geometric}$ be geometric quasi-sections at $x_{\mathfrak{s}},x_{\mathfrak{s}'}\in (\Compactification{X}_{\FieldAlgeClosure{k}})^{\closed}$, respectively.
	Assume that $k$ is perfect and that $r\geq 3$.
	Then
	\begin{equation*}
		x_{\mathfrak{s}}=x_{\mathfrak{s}'}
		\Rightarrow
		\mathfrak{s}\mathrm{R}\mathfrak{s}'
	\end{equation*}
	holds.
	If, in addition, we assume that
	\begin{equation*}
		r\geq \CardinalityOfSet{\left(\{x_{\mathfrak{s}},x_{\mathfrak{s}'}\}\cap E_{\FieldAlgeClosure{k}}\right)}+3,
	\end{equation*}
	then
	\begin{equation*}
		x_{\mathfrak{s}}=x_{\mathfrak{s}'}
		\Leftrightarrow
		\mathfrak{s}\mathrm{R}\mathfrak{s}'
	\end{equation*}
	holds.
	In particular, if $r\geq 5$, then the binary relation $\mathrm{R}$ is an equivalence relation on $\mathcal{QS}^{\geometric}$, and the natural map
	\begin{equation*}
		\mathcal{QS}^{\geometric}/\mathrm{R} \to (\Compactification{X}_{\FieldAlgeClosure{k}})^{\closed},
		\qquad
		[\mathfrak{s}]\mapsto x_{\mathfrak{s}}
	\end{equation*}
	is a bijection.
\end{lemma}

\begin{proof}
	Assume first that $x_{\mathfrak{s}}=x_{\mathfrak{s}'}$.
	Hence $\mathfrak{s}$ is cuspidal if and only if $\mathfrak{s}'$ is cuspidal.
	If both are cuspidal, then we set $I\coloneq I_{\mathfrak{s}}=I_{\mathfrak{s}'}$.
	If both are non-cuspidal, then we choose an arbitrary element $I\in \mathrm{Iner}^{\neq\{1\}}(\FinStepSolvQuoGeom{\EtFundGrpPi{}}{\abelian})$.
	Let $\mathcal{I}$ be a subset of $\mathrm{Iner}^{\neq\{1\}}(\FinStepSolvQuoGeom{\EtFundGrpPi{}}{\abelian})\setminus\{I\}$ with $\CardinalityOfSet{\mathcal{I}}=2$.
	Hence the common image of $x_{\mathfrak{s}}=x_{\mathfrak{s}'}$ in the corresponding two-cusp quotient is non-cuspidal.
	By Proposition~\ref{sepprop_quasisection}\ref{sepprop_quasisection_1}, non-cuspidal geometric quasi-sections of this quotient are identified, up to the action of $\FinStepSolvQuoGeom{\EtFundGrpPi{}}{\abelian}/N_{\mathcal{I}}$, with the closed points of the corresponding covering over the two-cusp curve.
	Hence $\rho_{\mathcal{I}}^{\mathrm{qs}}(\mathfrak{s})$ and $\rho_{\mathcal{I}}^{\mathrm{qs}}(\mathfrak{s}')$ lie in the same $\FinStepSolvQuoGeom{\EtFundGrpPi{}}{\abelian}/N_{\mathcal{I}}$-orbit.
	Since this holds for every such $\mathcal{I}$, we obtain $\mathfrak{s}\mathrm{R}\mathfrak{s}'$.

	Assume next that $\mathfrak{s}\mathrm{R}\mathfrak{s}'$.
	Choose $I$ as in Definition~\ref{decorel}.
	By the hypothesis, we may choose a subset $\mathcal{I}\subset \mathrm{Iner}^{\neq\{1\}}(\FinStepSolvQuoGeom{\EtFundGrpPi{}}{\abelian})$ with $\CardinalityOfSet{\mathcal{I}}=2$ such that $\mathcal{I}\cap\{I_{x_{\mathfrak{s}}},I_{x_{\mathfrak{s}'}},I\}=\emptyset$.
	Therefore, the images of $x_{\mathfrak{s}}$ and $x_{\mathfrak{s}'}$ in the corresponding two-cusp quotient are both non-cuspidal.
	On the other hand, by definition of $\mathrm{R}$, the quasi-sections $\rho_{\mathcal{I}}^{\mathrm{qs}}(\mathfrak{s})$ and $\rho_{\mathcal{I}}^{\mathrm{qs}}(\mathfrak{s}')$ lie in the same $\FinStepSolvQuoGeom{\EtFundGrpPi{}}{\abelian}/N_{\mathcal{I}}$-orbit.
	Again by Proposition~\ref{sepprop_quasisection}\ref{sepprop_quasisection_1}, their images in the corresponding two-cusp curve coincide.
	Since both $x_{\mathfrak{s}}$ and $x_{\mathfrak{s}'}$ lie outside the omitted cusps, this implies $x_{\mathfrak{s}}=x_{\mathfrak{s}'}$.
	This completes the proof.
\end{proof}

\begin{lemma}\label{prop:cuspcriterion-sim}
	We keep the notation of Notation~\ref{notation_basic1}.
	Assume that $k$ is perfect and that $r\geq 4$.
	For $\mathfrak{s}\in \mathcal{QS}^{\geometric}$, write
	\begin{equation*}
		[\mathfrak{s}]_{\mathrm{R}}
		\coloneq
		\{\mathfrak{t}\in \mathcal{QS}^{\geometric}\mid \mathfrak{s}\mathrm{R}\mathfrak{t}\}.
	\end{equation*}
	For $\mathfrak{s}\in \mathcal{QS}^{\geometric}$, we have
	\begin{equation*}
		\mathfrak{s}\notin \mathcal{QS}^{\geometric,\cuspidal}
		\iff
		\CardinalityOfSet{\left(\FinStepSolvQuo{\EtFundGrpGeom{}}{\abelian}\backslash[\mathfrak{s}]_{\mathrm{R}}\right)}=1.
	\end{equation*}
	In particular,
	\begin{equation*}
		\mathcal{QS}^{\geometric,\cuspidal}
		=
		\left\{
		\mathfrak{s}\in \mathcal{QS}^{\geometric}
		\ \middle|\
		\CardinalityOfSet{\left(\FinStepSolvQuo{\EtFundGrpGeom{}}{\abelian}\backslash[\mathfrak{s}]_{\mathrm{R}}\right)}>1
		\right\}.
	\end{equation*}
\end{lemma}

\begin{proof}
	We first assume that $\mathfrak{s}$ is non-cuspidal.
	Since $r\geq 4$, Lemma~\ref{decorellem} implies that
	\begin{equation*}
		[\mathfrak{s}]_{\mathrm{R}}
		=
		\{\mathfrak{t}\in \mathcal{QS}^{\geometric}\mid x_{\mathfrak{t}}=x_{\mathfrak{s}}\}.
	\end{equation*}
	On the other hand, Proposition~\ref{sepprop_quasisection}\ref{sepprop_quasisection_1} identifies $\mathcal{QS}^{\geometric,\noncuspidal}$ with $(\FinStepSolvQuoGeom{\UnivCov{X}}{\abelian})^{\closed}$ in a $\FinStepSolvQuo{\EtFundGrpGeom{}}{\abelian}$-equivariant way.
	Under this identification, the above set corresponds to the fiber of the natural map $(\FinStepSolvQuoGeom{\UnivCov{X}}{\abelian})^{\closed}\to (X_{\FieldAlgeClosure{k}})^{\closed}$ over $x_{\mathfrak{s}}$.
	Since $\FinStepSolvQuoGeom{\UnivCov{X}}{\abelian}\to X_{\FieldAlgeClosure{k}}$ is a Galois covering with Galois group $\FinStepSolvQuo{\EtFundGrpGeom{}}{\abelian}$, this fiber is a single $\FinStepSolvQuo{\EtFundGrpGeom{}}{\abelian}$-orbit.
	Hence $\CardinalityOfSet{\left(\FinStepSolvQuo{\EtFundGrpGeom{}}{\abelian}\backslash[\mathfrak{s}]_{\mathrm{R}}\right)}=1$.

	Next, we assume that $\mathfrak{s}$ is cuspidal.
	Let $(G,s)\in \mathfrak{s}$ be a representative with $s(G)\subset D_{\UnivCov{x}_{\mathfrak{s}}}$.
	Replacing $k$ by a finite extension if necessary, we may assume that $G=\AbsGalGrp{k}$ and that $\CardinalityOfSet{k}>2$.
	By Remark~\ref{rem:sections_coh_m1}\ref{rem:sections_coh_m1_2}, the set of sections $\AbsGalGrp{k}\to D_{\UnivCov{x}_{\mathfrak{s}}}$ modulo conjugation by $I_{\UnivCov{x}_{\mathfrak{s}}}$ is naturally identified with $H^{1}(\AbsGalGrp{k},I_{\UnivCov{x}_{\mathfrak{s}}})$.
	Since $I_{\UnivCov{x}_{\mathfrak{s}}}\cong \hat{\mathbb{Z}}^{p'}(1)$ and $k$ is perfect, we have $\CardinalityOfSet{H^{1}(\AbsGalGrp{k},I_{\UnivCov{x}_{\mathfrak{s}}})}>1$.
	Thus there exist two sections $s_{1},s_{2}\colon\AbsGalGrp{k}\to D_{\UnivCov{x}_{\mathfrak{s}}}$ which are not conjugate by $I_{\UnivCov{x}_{\mathfrak{s}}}$.
	Let $\mathfrak{s}_{1},\mathfrak{s}_{2}\in \mathcal{QS}^{\geometric}$ be the geometric quasi-sections represented by $(\AbsGalGrp{k},s_{1})$ and $(\AbsGalGrp{k},s_{2})$, respectively.
	Since $x_{\mathfrak{s}_{1}}=x_{\mathfrak{s}_{2}}=x_{\mathfrak{s}}$ and $r\geq 4$, Lemma~\ref{decorellem} implies that $\mathfrak{s}_{1},\mathfrak{s}_{2}\in [\mathfrak{s}]_{\mathrm{R}}$.
	We claim that $\mathfrak{s}_{1}$ and $\mathfrak{s}_{2}$ do not lie in the same $\FinStepSolvQuo{\EtFundGrpGeom{}}{\abelian}$-orbit.
	Indeed, assume that there exists $\delta\in \FinStepSolvQuo{\EtFundGrpGeom{}}{\abelian}$ such that $\delta\cdot \mathfrak{s}_{1}=\mathfrak{s}_{2}$.
	After restricting to a suitable open subgroup of $\AbsGalGrp{k}$, the images of the corresponding sections are contained in $D_{\delta\UnivCov{x}_{\mathfrak{s}}}\cap D_{\UnivCov{x}_{\mathfrak{s}}}$, and their common image in $\AbsGalGrp{k}$ is open.
	Hence Lemma~\ref{sepprop_e} implies that $\delta\UnivCov{x}_{\mathfrak{s}}=\UnivCov{x}_{\mathfrak{s}}$.
	Therefore $\delta\in D_{\UnivCov{x}_{\mathfrak{s}}}\cap \FinStepSolvQuo{\EtFundGrpGeom{}}{\abelian}=I_{\UnivCov{x}_{\mathfrak{s}}}$.
	This contradicts the choice of $s_{1}$ and $s_{2}$.
	Thus $\CardinalityOfSet{\left(\FinStepSolvQuo{\EtFundGrpGeom{}}{\abelian}\backslash[\mathfrak{s}]_{\mathrm{R}}\right)}>1$.
	This proves the assertion.
\end{proof}

\begin{lemma}\label{lemma:reconstruct-cuspidality-and-cusp-inertia}
	We keep the notation of Notation~\ref{notation_basic1}.
	Assume that $k$ is perfect and that $r\geq 4$.
	For $I\in \mathrm{Iner}^{\neq\{1\}}(\FinStepSolvQuoGeom{\EtFundGrpPi{}}{\abelian})$, we write $\pi_{I}\colon\FinStepSolvQuoGeom{\EtFundGrpPi{}}{\abelian}\twoheadrightarrow\FinStepSolvQuoGeom{\EtFundGrpPi{}}{\abelian}/I$ for the natural projection.
	We also write
	\begin{equation*}
		\pi_{I}^{\mathrm{qs}}\colon
		\mathcal{QS}
		\to
		\mathcal{QS}_{\FinStepSolvQuoGeom{\EtFundGrpPi{}}{\abelian}/I}
	\end{equation*}
	for the induced map.
	Let $\mathfrak{s}\in \mathcal{QS}^{\geometric,\cuspidal}$ be a geometric cuspidal quasi-section.
	Then $\pi_{I}^{\mathrm{qs}}(\mathfrak{s})$ is non-cuspidal if and only if $I=I_{\mathfrak{s}}$.
\end{lemma}

\begin{proof}
	The image of $\UnivCov{x}_{\mathfrak{s}}$ in the quotient corresponding to $\pi_{I_{\mathfrak{s}}}$ is non-cuspidal.
	On the other hand, if $I\neq I_{\mathfrak{s}}$, then the image of $\UnivCov{x}_{\mathfrak{s}}$ in the quotient corresponding to $\pi_{I}$ remains cuspidal.
	Furthermore, the quotient corresponding to $\pi_{I}$ is a curve of type $(0,r-1)$.
	Since $r-1\geq 3$, Proposition~\ref{sepprop_quasisection}\ref{sepprop_quasisection_2} implies that, for the quasi-section $\pi_{I}^{\mathrm{qs}}(\mathfrak{s})$, the property of being cuspidal or non-cuspidal is intrinsic.
	Therefore, the assertion follows.
\end{proof}

\begin{lemma}\label{cor2.13comm}
	We keep the notation of Notation~\ref{notation_basic1}.
	Assume that $k$ is perfect and that $r\geq 4$.
	Let $\UnivCov{x}$ be a closed point of $\FinStepSolvQuoGeom{\UnivCov{\Compactification{X}}}{\abelian}$, and let $x$ be its image in $\Compactification{X}$.
	Then there exists a geometric quasi-section at $\UnivCov{x}$.
	Moreover, for any $\mathfrak{s}\in \mathcal{QS}^{\geometric}$ at $\UnivCov{x}$ and any representative $(G,s)\in \mathfrak{s}$, we have
	\begin{equation*}
		D_{\UnivCov{x}}
		=
		\Commensurator{\FinStepSolvQuoGeom{\EtFundGrpPi{}}{\abelian}}
		{\GeneSubgrpTop{s(G), I_{\mathfrak{s}}}}.
	\end{equation*}
	Note that, by definition, we have $I_{\mathfrak{s}}=I_{\UnivCov{x}}$, which is trivial when $\UnivCov{x}$ is non-cuspidal.
\end{lemma}

\begin{proof}
	We first show that there exists a geometric quasi-section at $\UnivCov{x}$.
	If $\UnivCov{x}$ is non-cuspidal, then $s_{\UnivCov{x}}$ defined by Remark~\ref{rem:sections_coh_m1}\ref{rem:sections_coh_m1_1} is the desired geometric section at $\UnivCov{x}$.
	If $\UnivCov{x}$ is cuspidal, then the decomposition group at $\UnivCov{x}$ fits into an exact sequence
	\begin{equation*}
		1\to I_{\UnivCov{x}}\to D_{\UnivCov{x}}\to\AbsGalGrp{\kappa(x)}\to 1,
	\end{equation*}
	and this surjection admits a continuous section by \cite[Lemma~(2.9)]{MR1478817}.
	Hence, in either case, a geometric quasi-section at $\UnivCov{x}$ exists.

	Next, we show the second assertion.
	Let $\mathfrak{s}\in \mathcal{QS}^{\geometric}$ be at $\UnivCov{x}$, and let $(G,s)\in\mathfrak{s}$ be a representative.
	Since $s(G)\subset D_{\UnivCov{x}}$, we have
	\begin{equation*}
		\GeneSubgrpTop{s(G), I_{\UnivCov{x}}}
		\subset
		D_{\UnivCov{x}}.
	\end{equation*}
	On the other hand, the group $G$ is open in $\AbsGalGrp{\kappa(x)}=\mathrm{pr}(D_{\UnivCov{x}})$, and hence the subgroup $\GeneSubgrpTop{s(G), I_{\UnivCov{x}}}$ is also open in $D_{\UnivCov{x}}$.
	Therefore, Corollary~\ref{cor2.6comm} proves the second assertion.
\end{proof}

\begin{proposition}\label{bidecoabel}
	We keep the notation of Notation~\ref{notation_basic1}.
	Let
	\begin{equation*}
		\Phi\colon
		\FinStepSolvQuoGeom{\EtFundGrpPi{1}}{\metabelian}
		\xrightarrow{\ \sim\ }
		\FinStepSolvQuoGeom{\EtFundGrpPi{2}}{\metabelian}
	\end{equation*}
	be an isomorphism.
	Assume that $k_{1}$ is perfect, that $r_{1}\geq 5$, and that the bijection $\mathcal{QS}_{\FinStepSolvQuoGeom{\EtFundGrpPi{1}}{\abelian}}\xrightarrow{\sim}\mathcal{QS}_{\FinStepSolvQuoGeom{\EtFundGrpPi{2}}{\abelian}}$ induced by the induced isomorphism $\FinStepSolvQuoGeom{\Phi}{\abelian}$ preserves geometric quasi-sections.
	Then $\FinStepSolvQuoGeom{\Phi}{\abelian}$ preserves decomposition groups.
\end{proposition}

\begin{proof}
	By Lemma~\ref{geomreco1}\ref{finitereco}, $k_{1}$ is finite if and only if $k_{2}$ is finite.
	When $k_{1}$ is infinite, $\FinStepSolvQuo{\EtFundGrpGeom{1}}{\abelian}$ is a free $\hat{\mathbb{Z}}^{p_{1}'}$-module.
	Hence we have $p_{1}=p_{2}$ in this case.
	In particular, we obtain that $k_{2}$ is also perfect.
	By Lemma~\ref{geomreco2}\ref{rreco}, we have $r_{2}=r_{1}\geq 5$.
	By the assumption, the isomorphism $\FinStepSolvQuoGeom{\Phi}{\abelian}$ preserves geometric quasi-sections.
	Furthermore, Lemma~\ref{geomreco2_inerreco} implies that the isomorphism $\FinStepSolvQuoGeom{\Phi}{\abelian}$ preserves inertia groups.
	By Lemmas~\ref{geomreco2_inerreco} and~\ref{prop:cuspcriterion-sim}, the isomorphism $\FinStepSolvQuoGeom{\Phi}{\abelian}$ preserves whether a quasi-section is cuspidal or non-cuspidal.
	Here, the relation in Definition~\ref{decorel} is defined using the group-theoretical reconstruction of inertia groups (see~Lemma~\ref{geomreco2_inerreco}).
	Hence, by Lemma~\ref{lemma:reconstruct-cuspidality-and-cusp-inertia} and the hypothesis $r_{1}\geq 5$, for every $\mathfrak{s}\in \mathcal{QS}^{\geometric}_{1}$ we have
	\begin{equation*}
		\FinStepSolvQuoGeom{\Phi}{\abelian}(I_{\mathfrak{s}})
		=
		I_{\FinStepSolvQuoGeom{\Phi}{\abelian}(\mathfrak{s})}.
	\end{equation*}
	On the other hand, Lemma~\ref{cor2.13comm} implies that the decomposition group $D_{\mathfrak{s}}$ is determined group-theoretically from $\mathfrak{s}$ and, in the cuspidal case, from $I_{\mathfrak{s}}$.
	Hence $\FinStepSolvQuoGeom{\Phi}{\abelian}$ preserves decomposition groups.
\end{proof}

In the finite-field case, geometric sections admit a group-theoretical reconstruction (see~\cite[Proposition~(0.7)]{MR1478817}).
Hence the following corollary holds:

\begin{corollary}\label{bidecoabel_cor}
	We keep the notation of Notation~\ref{notation_basic1}.
	Let
	\begin{equation*}
		\Phi\colon
		\FinStepSolvQuoGeom{\EtFundGrpPi{1}}{\metabelian}
		\xrightarrow{\ \sim\ }
		\FinStepSolvQuoGeom{\EtFundGrpPi{2}}{\metabelian}
	\end{equation*}
	be an isomorphism.
	Assume that $k_{1}$ is finite and that $r_{1}\geq 5$.
	Then the induced isomorphism $\FinStepSolvQuoGeom{\Phi}{\abelian}$ preserves decomposition groups.
\end{corollary}

\begin{proof}
	By Lemma~\ref{geomreco1}\ref{finitereco}, $k_{2}$ is finite.
	The set $\mathcal{QS}_{i}^{\geometric}$ can be reconstructed group-theoretically from $\FinStepSolvQuoGeom{\EtFundGrpPi{i}}{\metabelian}$, as proved in \cite[Proposition~(0.7)]{MR1478817} (see~\cite[Lemma~2.8]{MR4745885} for more details in the abelian setting).
	Hence the assertion follows from Proposition~\ref{bidecoabel}.
\end{proof}

\begin{remark}\label{r=5remark}
	We keep the notation of Notation~\ref{notation_basic1}.
	The assumption $r_{1}\geq 5$ in Proposition~\ref{bidecoabel} and Corollary~\ref{bidecoabel_cor} comes from the reconstruction of the inertia subgroup $I_{\mathfrak{s}}$ attached to a geometric cuspidal quasi-section $\mathfrak{s}$.
	Indeed, to prove Proposition~\ref{bidecoabel}, we used the following steps:
	\begin{enumerate}[(i)]
		\item
		      We characterize group-theoretically whether a geometric quasi-section is cuspidal or non-cuspidal.
		      This is proved only under the assumption $r\geq 4$ by Lemma~\ref{prop:cuspcriterion-sim}.
		\item
		      We reconstruct $I_{\mathfrak{s}}$ using Lemma~\ref{lemma:reconstruct-cuspidality-and-cusp-inertia}, which studies the image of $\mathfrak{s}$ under the quotient by an inertia subgroup $I$.
		      The corresponding quotient curve is of type $(0,r-1)$.
		      Therefore, in order to apply the above group-theoretical characterization of cuspidality/non-cuspidality to the corresponding quotient curve, we need $r-1\geq 4$; equivalently, $r\geq 5$.
		\item
		      Finally, we reconstruct decomposition groups by Lemma~\ref{cor2.13comm}.
	\end{enumerate}
	If we could establish a group-theoretical characterization of cuspidal and non-cuspidal geometric quasi-sections also in the case $r=3$ (see Proposition~\ref{sepprop_quasisection}\ref{sepprop_quasisection_2}), then we could apply the same argument to Proposition~\ref{bidecoabel} and Corollary~\ref{bidecoabel_cor} when $r\geq 4$.
	However, at the time of writing, the author does not know whether such a characterization for $r=3$ is available.
\end{remark}

To prove the first main theorem of this paper, we need the following lemma:

\begin{lemma}\label{inj_decopre}
	We keep the notation of Notation~\ref{notation_basic1}.
	Assume that $k_{1}$ is finite and that $r\geq 2$.
	We write
	\begin{equation*}
		\Isom^{\mathrm{DP}}(\FinStepSolvQuoGeom{\EtFundGrpPi{1}}{\abelian},\FinStepSolvQuoGeom{\EtFundGrpPi{2}}{\abelian})
	\end{equation*}
	for the set of decomposition-preserving isomorphisms from $\FinStepSolvQuoGeom{\EtFundGrpPi{1}}{\abelian}$ to $\FinStepSolvQuoGeom{\EtFundGrpPi{2}}{\abelian}$.
	Then the natural map
	\begin{equation*}
		\Isom^{\mathrm{DP}}
		\bigl(
		\FinStepSolvQuoGeom{\EtFundGrpPi{1}}{\abelian},
		\FinStepSolvQuoGeom{\EtFundGrpPi{2}}{\abelian}
		\bigr)
		\to
		\Isom\bigl(
		\mathrm{Dec}(\FinStepSolvQuoGeom{\EtFundGrpPi{1}}{\abelian}),
		\mathrm{Dec}(\FinStepSolvQuoGeom{\EtFundGrpPi{2}}{\abelian})
		\bigr)
	\end{equation*}
	is injective.
\end{lemma}

\begin{proof}
	Let
	\begin{equation*}
		\Psi,\Psi'
		\in
		\Isom^{\mathrm{DP}}
		\bigl(
		\FinStepSolvQuoGeom{\EtFundGrpPi{1}}{\abelian},
		\FinStepSolvQuoGeom{\EtFundGrpPi{2}}{\abelian}
		\bigr)
	\end{equation*}
	be two isomorphisms that induce the same bijection on decomposition groups.
	By Lemma~\ref{geomreco1}\ref{finitereco}, the field $k_{2}$ is also finite.
	As in Lemma~\ref{geomreco1}\ref{geompireco}, the isomorphisms $\Psi$ and $\Psi'$ induce isomorphisms
	\begin{equation*}
		\overline{\Psi},\overline{\Psi'}\colon \AbsGalGrp{k_{1}}
		\xrightarrow{\ \sim\ }
		\AbsGalGrp{k_{2}}.
	\end{equation*}
	By Lemma~\ref{geomreco2}\ref{Frobreco}, both $\overline{\Psi}$ and $\overline{\Psi'}$ preserve the Frobenius element.
	Since $\AbsGalGrp{k_{1}}$ is topologically generated by $\Frobenius_{k_{1}}$, we obtain $\overline{\Psi}=\overline{\Psi'}$.
	Let $\UnivCov{u}_{1}\in \left(\FinStepSolvQuoGeom{\UnivCov{X_{1}}}{\abelian}\right)^{\closed}$, and let $u_{1}\in X_{1}^{\closed}$ be its image.
	Since $\Psi$ and $\Psi'$ induce the same map on decomposition groups, there exists a unique closed point $\UnivCov{u}_{2}\in \FinStepSolvQuoGeom{\UnivCov{X_{2}}}{\abelian}$ such that
	\begin{equation*}
		\Psi(D_{\UnivCov{u}_{1}})
		=
		D_{\UnivCov{u}_{2}}
		=
		\Psi'(D_{\UnivCov{u}_{1}}).
	\end{equation*}
	(Here, we used the assumption $r\geq 2$, see Remark~\ref{rem1.8}.)
	Let $u_{2}\in X_{2}^{\closed}$ be the image of $\UnivCov{u}_{2}$.
	Since $\UnivCov{u}_{1}$ and $\UnivCov{u}_{2}$ are non-cuspidal, the natural projections induce isomorphisms
	\begin{equation*}
		D_{\UnivCov{u}_{1}}
		\xrightarrow{\ \sim\ }
		\AbsGalGrp{\kappa(u_{1})},
		\qquad
		D_{\UnivCov{u}_{2}} \xrightarrow{\ \sim\ } \AbsGalGrp{\kappa(u_{2})}.
	\end{equation*}
	Let $\Frobenius_{\UnivCov{u}_{1}}\in D_{\UnivCov{u}_{1}}$ (resp.\ $\Frobenius_{\UnivCov{u}_{2}}\in D_{\UnivCov{u}_{2}}$) be the unique element mapping to the Frobenius element of $\AbsGalGrp{\kappa(u_{1})}$ (resp.\ $\AbsGalGrp{\kappa(u_{2})}$).
	Since $\overline{\Psi}=\overline{\Psi'}$ and both send $D_{\UnivCov{u}_{1}}$ onto $D_{\UnivCov{u}_{2}}$, we obtain
	\begin{equation*}
		\Psi(\Frobenius_{\UnivCov{u}_{1}})
		=
		\Frobenius_{\UnivCov{u}_{2}}
		=
		\Psi'(\Frobenius_{\UnivCov{u}_{1}}).
	\end{equation*}
	By the Chebotarev density theorem, the set
	\begin{equation*}
		\left\{
		\Frobenius_{\UnivCov{u}_{1}}
		\ \middle|\
		\UnivCov{u}_{1}\in
		\left(\FinStepSolvQuoGeom{\UnivCov{X_{1}}}{\abelian}\right)^{\closed}
		\right\}
	\end{equation*}
	is dense in $\FinStepSolvQuoGeom{\EtFundGrpPi{1}}{\abelian}$.
	Since $\Psi$ and $\Psi'$ are continuous and coincide on this dense subset, we conclude that $\Psi=\Psi'$.
\end{proof}

\section{Weak bi-anabelian results over finite fields}\label{subsectionfinweak}

In this section, we prove a pro-prime-to-$p$ version of the metabelian Grothendieck conjecture for affine hyperbolic genus-$0$ curves over finite fields (see Theorem~\ref{finGCweak} below).
The analogous conjecture for the geometrically maximal metabelian tame fundamental groups, rather than the pro-prime-to-$p$ quotients, was proved in \cite[Theorem~2.16]{MR4745885}.

\begin{notation}\label{addgroupdef}
	We keep the notation of Notation~\ref{notation_basic1}.
	Let $\UnivCov{v}$ be a closed point of $\FinStepSolvQuoGeom{\UnivCov{\Compactification{X}}}{\abelian}$ lying above a closed point $v$ of $X$.
	Assume that $k$ is finite.
	\begin{enumerate}[(1)]
		\item
		      Let $\ord_{v}\colon \FunctionField{X}^{\times} \twoheadrightarrow \mathbb{Z}$ be the valuation associated to $v$.
			We write 
			\begin{equation*}
			\mathcal{O}_{X,v}\coloneq\{a \in \FunctionField{X}\mid \ord_{v}(a)\geq 0\}
			\end{equation*}
			for the valuation ring of $\FunctionField{X}$ at $v$, and $\mathfrak{m}_{X,v}$ for the maximal ideal of $\mathcal{O}_{X,v}$. 
		\item
		      We write
		      \begin{equation*}
			      F_{\UnivCov{v}}
		      \end{equation*}
		      for the inverse image of $\GeneSubgrp{\Frobenius_{\kappa(v)}} \subset \AbsGalGrp{\kappa(v)}$ under the natural homomorphism $(D_{\UnivCov{v}})^{\abelian}\to \AbsGalGrp{\kappa(v)}$, where $\kappa(v)$ is the residue field of $X$ at $v$.
		      By \cite[Equation~(2.6)]{MR4745885}, this group depends only on $v$ up to isomorphism, so we also write $F_{v}$ instead of $F_{\UnivCov{v}}$.
		      We write
		      \begin{equation*}
			      \mathbb{A}_{\FinStepSolvQuoGeom{\EtFundGrpPi{}}{\abelian}}^{\times}
			      \coloneq
			      \prod_{v \in (\Compactification{X})^{\closed}}' F_{v},
		      \end{equation*}
		      the restricted direct product with respect to the subgroups
		      $\ker(F_{v} \to \AbsGalGrp{k})$.
		      We also write
		      \begin{equation*}
			      \FunctionField{\FinStepSolvQuoGeom{\EtFundGrpPi{}}{\abelian}}^{\times}
			      \coloneq
			      \ker(\mathbb{A}_{\FinStepSolvQuoGeom{\EtFundGrpPi{}}{\abelian}}^{\times}
			      \to
			      \FinStepSolvQuoGeom{\EtFundGrpPi{}}{\abelian}),
		      \end{equation*}
		      and
		      \begin{equation*}
			      \ord(\FinStepSolvQuoGeom{\EtFundGrpPi{}}{\abelian})_{v}\colon
			      \mathbb{A}_{\FinStepSolvQuoGeom{\EtFundGrpPi{}}{\abelian}}^{\times}
			      \to
			      F_{v}
			      \twoheadrightarrow
			      F_{v}/\ker(F_{v}
			      \to
			      \AbsGalGrp{k})
			      \xrightarrow{\ \sim\ }
			      \mathbb{Z}
		      \end{equation*}
		      for the composite, where the last isomorphism is determined by
		      $\Frobenius_{\kappa(v)} \mapsto 1$.
		      By restriction, we obtain a homomorphism
		      \begin{equation*}
			      \ord(\FinStepSolvQuoGeom{\EtFundGrpPi{}}{\abelian})_{v}\colon \FunctionField{\FinStepSolvQuoGeom{\EtFundGrpPi{}}{\abelian}}^{\times} \to \mathbb{Z}.
		      \end{equation*}
		      Finally, we define the multiplicative monoid
		      \begin{equation*}
			      \FunctionField{\FinStepSolvQuoGeom{\EtFundGrpPi{}}{\abelian}}
			      \coloneq
			      \FunctionField{\FinStepSolvQuoGeom{\EtFundGrpPi{}}{\abelian}}^{\times} \cup \{0\}.
		      \end{equation*}
	\end{enumerate}
\end{notation}

\begin{lemma}\label{reconstructionlemmaaddbase}
	We keep the notation of Notation~\ref{notation_basic1}.
	Assume that $k$ is finite and that $r\geq 3$.
	Then the natural homomorphism $\FunctionField{X}^{\times} \to \FunctionField{\FinStepSolvQuoGeom{\EtFundGrpPi{}}{\abelian}}^{\times}$ induced by the commutative diagram
	\begin{equation*}
		\vcenter{
		\xymatrix@C=42pt{
		1 \ar[r]
		& \FunctionField{X}^{\times} \ar[r] \ar[d]
		& \mathbb{A}_{\FunctionField{X}}^{\times} \ar[r] \ar@{->>}[d]
		& G_{\FunctionField{X}}^{\abelian} \ar@{->>}[d] \\
		1 \ar[r]
		& \FunctionField{\FinStepSolvQuoGeom{\EtFundGrpPi{}}{\abelian}}^{\times} \ar[r]
		& \mathbb{A}_{\FinStepSolvQuoGeom{\EtFundGrpPi{}}{\abelian}}^{\times} \ar[r]
		& \FinStepSolvQuoGeom{\EtFundGrpPi{}}{\abelian}
		}
		}
	\end{equation*}
	is an isomorphism.
	Moreover, under this isomorphism, the following hold:
	\begin{enumerate}[(1)]
		\item\label{globalcasereco2}
		      For every closed point $v$ of $\Compactification{X}$, the diagram
		      \begin{equation*}
			      \vcenter{
			      \xymatrix@C=45pt{
			      \FunctionField{X}^{\times} \ar@{->>}[r]^-{\ord_{v}} \ar[d]_{\cong}
			      & \mathbb{Z} \ar[d]^{=} \\
			      \FunctionField{\FinStepSolvQuoGeom{\EtFundGrpPi{}}{\abelian}}^{\times} \ar[r]^-{\ord(\FinStepSolvQuoGeom{\EtFundGrpPi{}}{\abelian})_{v}}
			      & \mathbb{Z}
			      }
			      }
		      \end{equation*}
		      is commutative.
		\item\label{globalcasereco3}
		      If $v$ is a closed point of $X$ (resp.~$E$), then $\mathcal{O}_{X,v}^{\times}$ (resp.~$1+\mathfrak{m}_{X,v}$) corresponds to
		      \begin{equation*}
			      \ker(\FunctionField{\FinStepSolvQuoGeom{\EtFundGrpPi{}}{\abelian}}^{\times} \to F_{v}).
		      \end{equation*}
		\item\label{globalcasereco4}
		      Let $L$ and $L^{\dag}$ be finite extensions of $k$ with $L\subset L^{\dag}$.
		      Then $\FinStepSolvQuoGeom{\EtFundGrpPi{X_{L}}}{\abelian}$ and $\FinStepSolvQuoGeom{\EtFundGrpPi{X_{L^{\dag}}}}{\abelian}$ are open subgroups of $\FinStepSolvQuoGeom{\EtFundGrpPi{}}{\abelian}$, and the diagram
		      \begin{equation*}
			      \vcenter{
			      \xymatrix@C=45pt{
			      \FunctionField{X_{L^{\dag}}}^{\times}
			      \ar[r]^-{\sim}
			      & \FunctionField{\FinStepSolvQuoGeom{\EtFundGrpPi{X_{L^{\dag}}}}{\abelian}}^{\times}
			      \\
			      \FunctionField{X_{L}}^{\times}
			      \ar[r]^-{\sim}
			      \ar@{^{(}->}[u]
			      & \FunctionField{\FinStepSolvQuoGeom{\EtFundGrpPi{X_{L}}}{\abelian}}^{\times}
			      \ar@{^{(}->}[u]
			      }
			      }
		      \end{equation*}
		      is commutative.
		      Here, the right-hand vertical arrow is induced by the transfer morphisms on the local factors.
		      In particular, we obtain an isomorphism
		      \begin{equation*}
			      \FunctionField{X_{\FieldAlgeClosure{k}}}\xrightarrow{\sim}
			      \plim_{L}\FunctionField{\FinStepSolvQuoGeom{\EtFundGrpPi{X_{L}}}{\abelian}}
		      \end{equation*}
	\end{enumerate}
\end{lemma}

\begin{proof}
	Since $g=0$, the $p$-rank is also $0$.
	In particular, the abelianization of the geometric tame fundamental group of $X$ is isomorphic to $\FinStepSolvQuo{\EtFundGrpGeom{}}{\abelian}$ (see~\cite[(2.5)]{MR1040998}).
	Moreover, for any closed point $\UnivCov{v}$ of $\FinStepSolvQuoGeom{\UnivCov{\Compactification{X}}}{\abelian}$, the decomposition group of the maximal geometrically abelian tame fundamental group of $X$ at any closed point above $\UnivCov{v}$ is isomorphic to $D_{\UnivCov{v},\FinStepSolvQuoGeom{\EtFundGrpPi{}}{\abelian}}$ by \cite[Lemma~1.10]{MR4745885}.
	Hence the same class-field-theoretic argument as in \cite[Lemma~2.13(1)(2) and Lemma~2.14]{MR4745885} proves the assertions.
\end{proof}

\begin{notation}\label{defihomset1}
	For any scheme $S_{i}$ and any scheme $T_{i}$ over $S_{i}$ ($i=1,2$), we write
	\begin{equation*}
		\Isom(T_{1}/S_{1},T_{2}/S_{2})
	\end{equation*}
	for the following set:
	\begin{equation*}
		\left\{(\tilde{F},F)\in \Isom_{\mathrm{Sch}}(T_{1},T_{2})\times \Isom_{\mathrm{Sch}}(S_{1},S_{2})\middle|\vcenter{
		\xymatrix{
		T_{1}\ar[r]^{\tilde{F}}\ar[d]& T_{2}\ar[d]\\
		S_{1}\ar[r]^{F} & S_{2}
		} }\text{ is commutative.}
		\right\}
	\end{equation*}
\end{notation}

\begin{theorem}\label{finGCweak}
	Let $i=1,2$.
	Let $k_{i}$ be a finite field of characteristic $p_{i}$.
	Let $(\Compactification{X}_{i},E_{i})$ be a smooth curve of type $(g_{i},r_{i})$ over $k_{i}$, and set $X_{i}\coloneq \Compactification{X}_{i}\setminus E_{i}$.
	(For the remaining notation, see the Introduction.)
	Assume that $g_{1}=0$ and that $r_{1}\geq 5$.
	Then there exists a map
	\begin{equation*}
		\mathcal{F}\colon
		\Isom(\FinStepSolvQuoGeom{\EtFundGrpPi{1}}{\metabelian},\FinStepSolvQuoGeom{\EtFundGrpPi{2}}{\metabelian})
		\to
		\Isom(X_{1,\FieldAlgeClosure{k}_{1}}/X_{1}, X_{2,\FieldAlgeClosure{k}_{2}}/X_{2})
	\end{equation*}
	such that the following triangular diagram commutes:
	\begin{equation}\label{finGCweak_tridiag}
		\vcenter{
		\xymatrix@C=45pt{
		\Isom(\FinStepSolvQuoGeom{\UnivCov{X_{1}}}{\metabelian}/X_{1},\FinStepSolvQuoGeom{\UnivCov{X_{2}}}{\metabelian}/X_{2})
		\ar@{^{(}->}[r]
		\ar@{->>}[d]
		&
		\Isom(\FinStepSolvQuoGeom{\EtFundGrpPi{1}}{\metabelian},\FinStepSolvQuoGeom{\EtFundGrpPi{2}}{\metabelian})
		\ar[d]
		\ar[ld]_{\mathcal{F}}
		\\
		\Isom(X_{1,\FieldAlgeClosure{k}_{1}}/X_{1},X_{2,\FieldAlgeClosure{k}_{2}}/X_{2})
		\ar@{^{(}->}[r]
		&
		\Isom(\FinStepSolvQuoGeom{\EtFundGrpPi{1}}{\abelian},\FinStepSolvQuoGeom{\EtFundGrpPi{2}}{\abelian})/\Inn(\FinStepSolvQuo{\EtFundGrpGeom{2}}{\abelian})
		}
		}
	\end{equation}
	Here, $\FinStepSolvQuoGeom{\UnivCov{X_{i}}}{\metabelian}$ denotes the geometrically maximal pro-prime-to-$p_{i}$ metabelian Galois covering of $\Compactification{X}_{i}$ that is \'etale over $X_{i}$, the lower horizontal arrows are the natural maps, and the right vertical arrow is induced by Lemma~\ref{geomreco1}\ref{geompireco}.
\end{theorem}

\begin{proof}
	The natural maps
	\begin{equation*}
		\Isom(\FinStepSolvQuoGeom{\UnivCov{X_{1}}}{\metabelian}/X_{1},\FinStepSolvQuoGeom{\UnivCov{X_{2}}}{\metabelian}/X_{2})
		\to
		\Isom(\FinStepSolvQuoGeom{\EtFundGrpPi{1}}{\metabelian},\FinStepSolvQuoGeom{\EtFundGrpPi{2}}{\metabelian})
	\end{equation*}
	and
	\begin{equation*}
		\Isom(X_{1,\FieldAlgeClosure{k}_{1}}/X_{1},X_{2,\FieldAlgeClosure{k}_{2}}/X_{2})
		\to
		\Isom(\FinStepSolvQuoGeom{\EtFundGrpPi{1}}{\abelian},\FinStepSolvQuoGeom{\EtFundGrpPi{2}}{\abelian})/\Inn(\FinStepSolvQuo{\EtFundGrpGeom{2}}{\abelian})
	\end{equation*}
	are injective; see~\cite[Lemma~4.9]{MR4745885}.
	We may assume that $\Isom(\FinStepSolvQuoGeom{\EtFundGrpPi{1}}{\metabelian},\FinStepSolvQuoGeom{\EtFundGrpPi{2}}{\metabelian})\neq\emptyset$.
	By \cite[Proposition~1.7(1)]{MR4745885}, we have $g_{2}=g_{1}(=0)$.
	Applying Lemma~\ref{geomreco2}\ref{rreco}\ref{nodereco} to both sides, we obtain $r_{2}=r_{1}(\geq 5)$, $p_{2}=p_{1}$, and $\CardinalityOfSet{k_{2}}=\CardinalityOfSet{k_{1}}$.
	Therefore, we omit the index $i$ from $g_{i}$, $r_{i}$, and $p_{i}$.

	We first construct the map $\mathcal{F}$.
	Let $\Phi$ be an element of $\Isom(\FinStepSolvQuoGeom{\EtFundGrpPi{1}}{\metabelian},\FinStepSolvQuoGeom{\EtFundGrpPi{2}}{\metabelian})$.
	By \cite[Proposition~(0.7)]{MR1478817} (or \cite[Lemma~2.8]{MR4745885}), the bijection $\mathcal{QS}_{\FinStepSolvQuoGeom{\EtFundGrpPi{1}}{\abelian}}\xrightarrow{\sim}\mathcal{QS}_{\FinStepSolvQuoGeom{\EtFundGrpPi{2}}{\abelian}}$ induced by $\FinStepSolvQuoGeom{\Phi}{\abelian}$ preserves geometric quasi-sections.
	Moreover, Corollary~\ref{bidecoabel_cor} implies that $\FinStepSolvQuoGeom{\Phi}{\abelian}$ preserves decomposition groups.
	Therefore, as noted in Remark~\ref{rem1.8}, we obtain the following bijections arranged in a commutative diagram:
	\begin{equation*}
		\xymatrix@C=45pt{
		\mathcal{QS}_{\FinStepSolvQuoGeom{\EtFundGrpPi{1}}{\abelian}}^{\geometric}
		\ar[r]^{\sim}
		\ar@{->>}[d]
		&
		\mathcal{QS}_{\FinStepSolvQuoGeom{\EtFundGrpPi{2}}{\abelian}}^{\geometric}
		\ar@{->>}[d]
		\\
		\mathrm{Dec}\left(\FinStepSolvQuoGeom{\EtFundGrpPi{1}}{\abelian}\right)
		\ar[r]^{\sim}
		&
		\mathrm{Dec}\left(\FinStepSolvQuoGeom{\EtFundGrpPi{2}}{\abelian}\right)
		\\
		\left(\FinStepSolvQuoGeom{\UnivCov{\Compactification{X}_{1}}}{\abelian}\right)^{\closed}
		\ar[u]^{\cong}
		\ar[r]^{\sim}
		&
		\left(\FinStepSolvQuoGeom{\UnivCov{\Compactification{X}_{2}}}{\abelian}\right)^{\closed}
		\ar[u]^{\cong}
		}
	\end{equation*}
	By Lemma~\ref{geomreco2}\ref{Frobreco}, the induced isomorphism $\AbsGalGrp{k_{1}}\xrightarrow{\sim} \AbsGalGrp{k_{2}}$ preserves the Frobenius element.
	Therefore, by definition, $\Phi$ induces an isomorphism
	\begin{equation*}
		\mathbb{A}_{\FinStepSolvQuoGeom{\EtFundGrpPi{1}}{\abelian}}^{\times}
		\xrightarrow{\ \sim\ }
		\mathbb{A}_{\FinStepSolvQuoGeom{\EtFundGrpPi{2}}{\abelian}}^{\times}.
	\end{equation*}
	Hence Lemma~\ref{reconstructionlemmaaddbase}\ref{globalcasereco4} implies that $\Phi$ induces a multiplicative monoid isomorphism
	\begin{equation}\label{geomisommap}
		\FunctionField{X_{1,\FieldAlgeClosure{k_{1}}}}\to \FunctionField{X_{2,\FieldAlgeClosure{k_{2}}}}.
	\end{equation}
	Moreover, the multiplicative monoid isomorphism in \eqref{geomisommap} is additive.
	Indeed, this follows from Lemma~\ref{reconstructionlemmaaddbase}\ref{globalcasereco2}\ref{globalcasereco3} and \cite[Lemma~4.7]{MR1478817} (or \cite[Lemma~2.15]{MR4745885}).
	Hence $\Phi$ induces a scheme isomorphism
	\begin{equation*}
		\Compactification{X}_{1,\FieldAlgeClosure{k}_{1}}
		\xrightarrow{\ \sim\ }
		\Compactification{X}_{2,\FieldAlgeClosure{k}_{2}}.
	\end{equation*}
	By restricting the isomorphism \eqref{geomisommap}, we have an isomorphism $\FunctionField{X_{1}}\xrightarrow{\sim}\FunctionField{X_{2}}$.
	Hence we also obtain a scheme isomorphism
	\begin{equation*}
		\Compactification{X}_{1}
		\xrightarrow{\ \sim\ }
		\Compactification{X}_{2}.
	\end{equation*}
	Since $\Phi$ preserves cuspidal decomposition groups, by restricting the above scheme isomorphisms we obtain
	\begin{equation*}
		\UnivCov{u}\colon X_{1,\FieldAlgeClosure{k}_{1}}\xrightarrow{\sim} X_{2,\FieldAlgeClosure{k}_{2}}
		\qquad
		\text{and}
		\qquad
		u\colon X_{1}\xrightarrow{\sim} X_{2}.
	\end{equation*}
	The pair of these isomorphisms is the desired element $\mathcal{F}(\Phi)$.

	Next, we show the commutativity of the lower-right triangle of \eqref{finGCweak_tridiag}.
	By the construction of $\mathcal{F}$, $\mathcal{F}(\Phi)$ and $\Phi$ induce the same bijection $\mathrm{Dec}(\FinStepSolvQuoGeom{\EtFundGrpPi{1}}{\abelian})\xrightarrow{\sim}\mathrm{Dec}(\FinStepSolvQuoGeom{\EtFundGrpPi{2}}{\abelian})$ up to conjugacy by $\FinStepSolvQuo{\EtFundGrpGeom{2}}{\abelian}$.
	By Lemma~\ref{inj_decopre}, the map
	\begin{equation*}
		\Isom^{\mathrm{DP}}(\FinStepSolvQuoGeom{\EtFundGrpPi{1}}{\abelian},\FinStepSolvQuoGeom{\EtFundGrpPi{2}}{\abelian})/\FinStepSolvQuo{\EtFundGrpGeom{2}}{\abelian}
		\to
		\Isom(\mathrm{Dec}(\FinStepSolvQuoGeom{\EtFundGrpPi{1}}{\abelian}),\mathrm{Dec}(\FinStepSolvQuoGeom{\EtFundGrpPi{2}}{\abelian}))/\FinStepSolvQuo{\EtFundGrpGeom{2}}{\abelian}
	\end{equation*}
	is injective.
	This proves the commutativity of the lower-right triangle.
	After composing with the second lower horizontal arrow of \eqref{finGCweak_tridiag}, the commutativity of the right-hand triangle follows immediately.
	This completes the proof.
\end{proof}

\begin{corollary}\label{finitrelative}
	We keep the notation and assumptions of Theorem~\ref{finGCweak}.
	Assume that
	\begin{equation*}
		\FinStepSolvQuoGeom{\EtFundGrpPi{1}}{\metabelian}
		\cong
		\FinStepSolvQuoGeom{\EtFundGrpPi{2}}{\metabelian}
	\end{equation*}
	as profinite groups.
	Then $X_{1}$ and $X_{2}$ are isomorphic as schemes.
\end{corollary}

\begin{proof}
	The assertion follows immediately from Theorem~\ref{finGCweak}.
\end{proof}

We introduce the following notation for the next section:

\begin{definition}\label{frobconstreco}
	We keep the notation of Notation~\ref{notation_basic1}.
	Assume that $k$ is finite and that $r\geq 5$.
	We define a multiplicative monoid
	\begin{equation*}
		k(\FinStepSolvQuoGeom{\EtFundGrpPi{}}{\abelian})
		\coloneq
		\left\{
		a\in
		\FunctionField{\FinStepSolvQuoGeom{\EtFundGrpPi{}}{\abelian}}^{\times}
		\ \middle|\
		\ord(\FinStepSolvQuoGeom{\EtFundGrpPi{}}{\abelian})_{v}(a)=0
		\text{ for every closed point }v\text{ of }X
		\right\}
		\cup
		\{0\}.
	\end{equation*}
	Passing to the direct limit over all finite extensions $L/k$, we define
	\begin{equation*}
		\FieldAlgeClosure{k}(\FinStepSolvQuoGeom{\EtFundGrpPi{}}{\abelian})
		\coloneq
		\varinjlim_{L/k}
		k(\FinStepSolvQuoGeom{\EtFundGrpPi{X_{L}}}{\abelian}).
	\end{equation*}
\end{definition}

\begin{remark}\label{remalphaphi}
	We keep the notation of Notation~\ref{notation_basic1}.
	Assume that $k$ is finite and that $r\geq 5$.
	Under the isomorphism in Lemma~\ref{reconstructionlemmaaddbase}\ref{globalcasereco2}, the multiplicative monoids $k(\FinStepSolvQuoGeom{\EtFundGrpPi{}}{\abelian})$ and $\FieldAlgeClosure{k}(\FinStepSolvQuoGeom{\EtFundGrpPi{}}{\abelian})$ are canonically identified with $k$ and $\FieldAlgeClosure{k}$, respectively, as multiplicative monoids.
	By construction and Lemma~\ref{reconstructionlemmaaddbase}\ref{globalcasereco4}, any isomorphism $\Phi\colon\FinStepSolvQuoGeom{\EtFundGrpPi{1}}{\metabelian}\xrightarrow{\sim}\FinStepSolvQuoGeom{\EtFundGrpPi{2}}{\metabelian}$ induces a commutative diagram of multiplicative monoids
	\begin{equation*}
		\xymatrix{
		k(\FinStepSolvQuoGeom{\EtFundGrpPi{1}}{\abelian})
		\ar[r]^-{\sim}
		\ar@{^{(}->}[d]
		&
		k(\FinStepSolvQuoGeom{\EtFundGrpPi{2}}{\abelian})
		\ar@{^{(}->}[d]
		\\
		\FieldAlgeClosure{k}(\FinStepSolvQuoGeom{\EtFundGrpPi{1}}{\abelian})
		\ar[r]^-{\sim}
		&
		\FieldAlgeClosure{k}(\FinStepSolvQuoGeom{\EtFundGrpPi{2}}{\abelian}).
		}
	\end{equation*}
	Moreover, the proof of Theorem~\ref{finGCweak} shows that the lower horizontal isomorphism arises from a field isomorphism.
	Therefore, when $k_{1}=k_{2}$, the isomorphism $\Phi$ induces an element of $\AbsGalGrp{\mathbb{F}_{p}}$.
	Thus there exists a unique element $\alpha(\Phi)\in \hat{\mathbb{Z}}$ such that the induced automorphism is $\Frobenius_{\mathbb{F}_{p}}^{\,\alpha(\Phi)}$.
\end{remark}

\section{Specialization}\label{subsectionspecialization}

In this section, we study specialization morphisms for smooth curves over henselian regular local rings.
They relate the geometric fundamental groups of the generic fiber and the special fiber, and will be used later to extend Theorem~\ref{finGCweak} to the metabelian Grothendieck conjecture over finitely generated fields.
Here, we define smooth curves over arbitrary schemes, not only over spectra of fields.
\begin{definition}
	Let $S$ be a scheme. Let $\Compactification{\mathcal{X}}$ be a scheme over $S$, and let $\mathcal{E}$ be a (possibly empty) closed subscheme of $\Compactification{\mathcal{X}}$.
	We say that the pair $(\Compactification{\mathcal{X}},\mathcal{E})$ is a \emph{smooth curve of type $(g,r)$} over $S$ if the following conditions~(a)--(b) hold:
	\begin{enumerate}[(a)]
		\item
		      $\Compactification{\mathcal{X}} \to S$ is a smooth proper morphism of relative dimension one whose geometric fibers are connected of genus $g$;
		\item
		      the composition of $\mathcal{E}\hookrightarrow \Compactification{\mathcal{X}}\to S$ is finite, \'etale, and of degree $r$.
	\end{enumerate}
	If there is no risk of confusion, we call the complement $\mathcal{X}\coloneq \Compactification{\mathcal{X}}\setminus\mathcal{E}$ a \emph{smooth curve of type $(g,r)$} over $S$.
	We call $\mathcal{X}$ \emph{hyperbolic} if $2 - 2g - r < 0$ (in other words, $(g, r) \neq (0, 0), (0, 1), (0, 2), (1, 0)$).
\end{definition}

We first consider the specialization morphism.
Let $R$ be a henselian regular local ring with the field of fractions $K$.
Let $\mathfrak{X}$ be a smooth curve of type $(g,r)$ over $\Spec(R)$.
We have the following basic lemma:

\begin{lemma}[{\cite[Expos\'e XIII, Lemme 4.2, Exemples 4.4]{MR0354651}}]\label{sgabasicexaxt}
	In the above notation, for any point $u\in \Spec(R)$ and any geometric point $\overline{u}$ of $\Spec(R)$ lying over $u$, the sequence of natural morphisms
	\begin{equation*}
		\EtFundGrpGeom{\mathfrak{X}_{u}}^{\tame}
		\to
		\EtFundGrpPi{\mathfrak{X}}^{\tame}
		\to
		\EtFundGrpTameWithPt{\Spec(R)}{\overline{u}}
		\to 1
	\end{equation*}
	is exact.
\end{lemma}

Let $s\in \Spec(R)$ be the closed point, and let $\eta\in \Spec(R)$ be the generic point.
Write $\kappa(s)$ for the residue field at $s$ of characteristic $p_{s}$ $(\geq 0)$.
By \cite[Expos\'e X, Corollaire 3.9]{MR0354651}, we have the specialization morphism
\begin{equation*}
	\EtFundGrpGeom{\mathfrak{X}_{\eta}}^{\tame}
	\to
	\EtFundGrpGeom{\mathfrak{X}_{s}}^{\tame},
\end{equation*}
and it is an isomorphism after taking the pro-prime-to-$p_{s}$ quotients of both sides.
Therefore, by taking the geometrically maximal metabelian pro-prime-to-$p_{s}$ quotient, we obtain the following commutative diagram of exact rows:

\begin{equation}\label{commdaigspaci2}
	\vcenter{\xymatrix{
	1
	\ar[r]
	&
	\EtFundGrpGeom{\mathfrak{X}_{\eta}}^{p_{s}',\metabelian}
	\ar[r]\ar[dd]^-{\cong}_-{\mathrm{sp}_s}
	&
	\EtFundGrpPi{\mathfrak{X}_{\eta}}^{(p_{s}',\metabelian)}
	\ar[r]^-{\rho}\ar@{->>}[dd]_-{\tilde{\mathrm{sp}}_s}
	&
	\AbsGalGrp{K}
	\ar[r]\ar@{->>}[d]
	&
	1
	\\
	&
	&
	&
	\EtFundGrpTameWithPt{\Spec(R)}{\overline{\eta}}
	&
	\\
	1
	\ar[r]
	&
	\FinStepSolvQuo{\EtFundGrpGeom{\mathfrak{X}_{s}}}{\metabelian}
	\ar[r]
	&
	\FinStepSolvQuoGeom{\EtFundGrpPi{\mathfrak{X}_{s}}}{\metabelian}
	\ar[r]
	&
	\AbsGalGrp{\kappa(s)}
	\ar[r]\ar[u]^{\cong}
	&
	1
	\\
	}}
\end{equation}

The following lemma gives a group-theoretical description of the kernel of the above morphism $\tilde{\mathrm{sp}}_s$.

\begin{lemma}\label{recobasicfiberproduct}
	Consider a commutative diagram of profinite groups with exact rows
	\begin{equation*}
		\vcenter{\xymatrix{
		1 \ar[r] & A \ar[r] \ar[d]^{\cong} & B_{1} \ar[r]^{\rho} \ar[d]^{\rho_{B}} & C_{1} \ar[r] \ar[d]^{\rho_{C}} & 1 \\
		1 \ar[r] & A \ar[r] & B_{2} \ar[r] & C_{2} \ar[r] & 1.
		}}
	\end{equation*}
	Put $I\coloneq \ker(\rho_{C})$, and let $B_{1}\to \Aut(A)$ be the action by conjugation.
	Assume that $A$ is center-free.
	Then
	\begin{equation*}
		\ker(\rho_{B})=\rho^{-1}(I)\cap \ker(B_{1}\to \Aut(A))
	\end{equation*}
	holds.
\end{lemma}

\begin{proof}
	We consider the following commutative diagram with exact rows:
	\begin{equation*}
		\vcenter{\xymatrix{
		1 \ar[r] & A \ar[r] \ar[d]^{\cong} & B_{1} \ar[r]^{\rho} \ar[d]^{\rho_{B}} & C_{1} \ar[r] \ar[d]^{\rho_{C}} & 1 \\
		1 \ar[r] & A \ar[d]^{\cong}\ar[r] & B_{2} \ar[d]\ar[r] & C_{2} \ar[r] \ar[d]& 1.\\
		1 \ar[r] & \Inn(A) \ar[r] & \Aut(A) \ar[r] & \Out(A) \ar[r] & 1.
		}}
	\end{equation*}
	A diagram chase shows that
	\begin{equation*}
		\ker(\rho_{B})\subset\rho^{-1}(I)\cap \ker(B_{1}\to \Aut(A)).
	\end{equation*}
	We show the converse inclusion.
	Let $x\in \rho^{-1}(I)\cap \ker(B_{1}\to \Aut(A))$.
	By the definition, the image of $x$ in $C_{2}$ is trivial.
	Hence $\rho_{B}(x)\in A$.
	On the other hand, by the definition, the image of $x$ in $\Aut(A)$ is also trivial.
	Therefore, we have $\rho_{B}(x)\in \CenterSubgrp{A}$.
	Since $A$ is center-free, we obtain $\rho_{B}(x)=1$; equivalently, $x\in \ker(\rho_{B})$.
	This completes the proof.
\end{proof}

\begin{proposition}\label{sp_prop_reco}
	We keep the above notation.
	Assume that $\mathfrak{X}_{\eta}$ is hyperbolic.
	Let $I$ be the inertia subgroup of $\AbsGalGrp{K}$, and let
	\begin{equation*}
		f\colon
		\EtFundGrpPi{\mathfrak{X}_{\eta}}^{(p_{s}',\metabelian)}
		\to
		\Aut\left(\EtFundGrpGeom{\mathfrak{X}_{\eta}}^{p_{s}',\metabelian}\right)
	\end{equation*}
	be the action by conjugation.
	Then we have
	\begin{equation*}
		\ker\left(\tilde{\mathrm{sp}}_{s}\right)=\rho^{-1}(I)\cap \ker(f)
	\end{equation*}
\end{proposition}

\begin{proof}
	The center-freeness of $\EtFundGrpGeom{\mathfrak{X}_{\eta}}^{p_{s}',\metabelian}$ follows from \cite[Theorem~2.9]{yamaguchi2026centerfreenessfinitestepsolvablegroups}; in the present situation, however, \cite[Proposition~1.1.6]{MR4578639} already suffices.
	Thus, the assertion follows immediately from Lemma~\ref{recobasicfiberproduct} applied to the diagram \eqref{commdaigspaci2}.
\end{proof}

\begin{remark}\label{sprem}
	In general, the specialization morphism $\EtFundGrpGeom{\mathfrak{X}_{\eta}}^{\tame}
		\twoheadrightarrow
		\EtFundGrpGeom{\mathfrak{X}_{s}}^{\tame}$ is not injective.
	However, by the metabelian good reduction criterion established in \cite[Theorem~3.8]{MR4745885}, the kernel of the specialization morphism can be reconstructed group-theoretically from $\FinStepSolvQuoGeom{(\EtFundGrpPi{\mathfrak{X}_{\eta}}^{\tame})}{\metabelian}
		\to \AbsGalGrp{K}$ together with the inertia subgroup $I$ of $\AbsGalGrp{K}$; see \cite[Section~3]{MR4745885} for more details.
\end{remark}

\begin{corollary}\label{localspecializationiso}
	Let $R$ be a henselian regular local ring with field of fractions $K$.
	Let $\eta\in \Spec(R)$ be the generic point, and let $s$ be the closed point of $\Spec(R)$.
	Write $p_{s}$ for the characteristic of $\kappa(s)$.
	Let $i=1,2$.
	Let $\mathfrak{X}_{i}$ be a smooth curve over $\Spec(R)$.
	Assume that $\mathfrak{X}_{1,\eta}$ is hyperbolic.
	Let
	\begin{equation*}
		\Phi\colon
		\EtFundGrpPi{\mathfrak{X}_{1,\eta}}^{(\metabelian)}
		\xrightarrow{\sim}
		\EtFundGrpPi{\mathfrak{X}_{2,\eta}}^{(\metabelian)}
	\end{equation*}
	be a $\AbsGalGrp{K}$-isomorphism.
	Then $\Phi$ induces a unique $\AbsGalGrp{\kappa(s)}$-isomorphism
	\begin{equation*}
		\Phi_{s}\colon
		\FinStepSolvQuoGeom{\EtFundGrpPi{\mathfrak{X}_{1,s}}}{\metabelian}
		\xrightarrow{\sim}
		\FinStepSolvQuoGeom{\EtFundGrpPi{\mathfrak{X}_{2,s}}}{\metabelian}
	\end{equation*}
	such that the following diagram commutes:
	\begin{equation*}
		\vcenter{
		\xymatrix{
		\EtFundGrpPi{\mathfrak{X}_{1,\eta}}^{(\metabelian)}
		\ar[r]^-{\Phi}
		\ar@{->>}[d]
		&
		\EtFundGrpPi{\mathfrak{X}_{2,\eta}}^{(\metabelian)}
		\ar@{->>}[d]
		\\
		\EtFundGrpPi{\mathfrak{X}_{1,\eta}}^{(p_{s}',\metabelian)}
		\ar@{->>}[d]_-{\tilde{\mathrm{sp}}_{1,s}}
		&
		\EtFundGrpPi{\mathfrak{X}_{2,\eta}}^{(p_{s}',\metabelian)}
		\ar@{->>}[d]_-{\tilde{\mathrm{sp}}_{2,s}}
		\\
		\FinStepSolvQuoGeom{\EtFundGrpPi{\mathfrak{X}_{1,s}}}{\metabelian}
		\ar[r]^-{\Phi_{s}}
		&
		\FinStepSolvQuoGeom{\EtFundGrpPi{\mathfrak{X}_{2,s}}}{\metabelian}
		}
		}
	\end{equation*}
\end{corollary}

\begin{proof}
	We know that a smooth curve over a field is hyperbolic if and only if the maximal pro-$\ell$ metabelian quotient of its geometric tame fundamental group is non-abelian for some prime $\ell$ different from the characteristic of the field.
	Hence $\mathfrak{X}_{2,\eta}$ is also hyperbolic.
	For $i=1,2$, the specialization morphism $\mathrm{sp}_{i,s}$ induces an isomorphism after taking the maximal pro-prime-to-$p_{s}$ quotient.
	Thus the surjection $\tilde{\mathrm{sp}}_{i,s}$ identifies $\FinStepSolvQuoGeom{\EtFundGrpPi{\mathfrak{X}_{i,s}}}{\metabelian}$ with the quotient of $\FinStepSolvQuoGeom{\EtFundGrpPi{\mathfrak{X}_{i,\eta}}}{p_{s}',\metabelian}$ by $\ker(\tilde{\mathrm{sp}}_{i,s})$.
	By Proposition~\ref{sp_prop_reco}, the subgroup $\ker(\tilde{\mathrm{sp}}_{i,s})$ is determined by the projection to $\AbsGalGrp{K}$, the inertia subgroup of $\AbsGalGrp{K}$, and the conjugation action on the geometric metabelian quotient.
	Since $\Phi$ is a $\AbsGalGrp{K}$-isomorphism, it identifies these data for $i=1,2$.
	Therefore $\Phi$ descends uniquely to an isomorphism
	\begin{equation*}
		\Phi_{s}\colon
		\FinStepSolvQuoGeom{\EtFundGrpPi{\mathfrak{X}_{1,s}}}{\metabelian}
		\xrightarrow{\sim}
		\FinStepSolvQuoGeom{\EtFundGrpPi{\mathfrak{X}_{2,s}}}{\metabelian}.
	\end{equation*}
	Since $\Phi$ is over $\AbsGalGrp{K}$, the induced isomorphism is over $\AbsGalGrp{\kappa(s)}$.
	The commutativity of the diagram is immediate from the construction, and the uniqueness follows from the surjectivity of vertical arrows.
\end{proof}

\section{Weak bi-anabelian results over finitely generated fields}\label{subsectionfingeneweak}

In this section, we prove a pro-prime-to-$p$ version of the metabelian Grothendieck conjecture for genus-$0$ hyperbolic curves over $k$.

We first recall the moduli scheme of ordered hyperbolic genus-$0$ curves.
We write $\mathcal{M}_{g,r}$ for the moduli stack of smooth curves of genus $g$ equipped with $r$ disjoint ordered sections.
By~\cite{MR0702953}, the stack $\mathcal{M}_{0,r}$ is a separated Deligne--Mumford stack over $\Spec(\mathbb{Z})$.
Since an automorphism of a smooth proper genus-$0$ curve that fixes three ordered marked points is trivial, when $r\geq 3$ the stack $\mathcal{M}_{0,r}$ is represented by a scheme.
We denote this scheme by $M_{0,r}$.

\begin{definition}\label{defcyclotome}
	We keep the notation of Notation~\ref{notation_basic1}.
	Assume that $r\geq 2$.
	We define the cyclotome associated to $X$ by
	\begin{equation}\label{cyclotome}
		\Lambda\coloneq\Lambda_{X}
		\coloneq
		\ker\left(
		\left(
			\bigoplus_{I\in \mathrm{Iner}(\FinStepSolvQuo{\EtFundGrpGeom{X}}{\abelian})} I
			\right)
		\to
		\EtFundGrpGeom{X}^{\abelian}
		\right).
	\end{equation}
	By \eqref{wfseq}, we have an isomorphism $\Lambda_{X}\xrightarrow{\sim}\hat{\mathbb{Z}}^{p'}(1)$ of $\AbsGalGrp{k}$-modules.
\end{definition}

\begin{lemma}\label{lem:global-frobenius-discrepancy}
	Let $S$ be a regular integral scheme of finite type over $\Spec(\mathbb{Z})$, and let $\eta$ be its generic point.
	Let $p\geq 0$ be the characteristic of the residue field $\kappa(\eta)$ of $S$ at $\eta$, and let $r\geq 5$.
	For $i=1,2$, let $\zeta_{i}\colon S\to \mathcal{M}_{0,r}$ be a morphism, let $(\Compactification{\mathcal{X}_{i}},\mathcal{E}_{i})$ be the corresponding smooth curve of type $(0,r)$ over $S$, and set $\mathcal{X}_{i}\coloneq \Compactification{\mathcal{X}_{i}}\setminus \mathcal{E}_{i}$.
	Let
	\begin{equation*}
		\Phi\colon
		\FinStepSolvQuoGeom{\EtFundGrpPi{\mathcal{X}_{1,\eta}}}{\metabelian}
		\xrightarrow{\ \sim\ }
		\FinStepSolvQuoGeom{\EtFundGrpPi{\mathcal{X}_{2,\eta}}}{\metabelian}
	\end{equation*}
	be a $\AbsGalGrp{\kappa(\eta)}$-isomorphism, and let $\lambda_{\Phi}\colon \Lambda_{\mathcal{X}_{1,\eta}}\xrightarrow{\sim}\Lambda_{\mathcal{X}_{2,\eta}}$ be the canonical isomorphism induced by $\Phi$.
	For every closed point $s\in S$, we write $p_{s}$ for the characteristic of $\kappa(s)$ and
	\begin{equation*}
		\Phi_{s}\colon
		\FinStepSolvQuoGeom{\EtFundGrpPi{\mathcal{X}_{1,s}}}{\metabelian}
		\xrightarrow{\ \sim\ }
		\FinStepSolvQuoGeom{\EtFundGrpPi{\mathcal{X}_{2,s}}}{\metabelian}
	\end{equation*}
	for the $\AbsGalGrp{\kappa(s)}$-isomorphism induced by $\Phi$ (see Corollary~\ref{localspecializationiso}).
	Then the following hold:
	\begin{enumerate}[(1)]
		\item\label{lem:global-frobenius-discrepancy-char0}
		      Assume that $p=0$.
		      Then $\zeta_{1}=\zeta_{2}$.
		      Moreover, $\alpha(\Phi_{s})=0$ for every closed point $s\in S$, and, under the canonical identifications $\Lambda_{\mathcal{X}_{i,\eta}}\xrightarrow{\sim}\hat{\mathbb{Z}}(1)$, the isomorphism $\lambda_{\Phi}$ is the identity.

		\item\label{lem:global-frobenius-discrepancy-charp}
		      Assume that $p>0$ and that $\dim(\overline{\zeta_{1}(S)})>0$.
		      Then there exists a unique pair $(n_{1},n_{2})$ of non-negative integers with $n_{1}n_{2}=0$ such that
		      \begin{equation*}
			      \zeta_{1}\circ \Frobenius_{S}^{n_{1}}
			      =
			      \zeta_{2}\circ \Frobenius_{S}^{n_{2}},
		      \end{equation*}
		      where $\Frobenius_{S}$ is the relative Frobenius morphism of $S$ over $\mathbb{F}_{p}$.
		      Moreover, the equality $n_{2}-n_{1}=\alpha(\Phi_{s})$ holds for every closed point $s\in S$ (see Remark~\ref{remalphaphi}), and, under the canonical identifications $\Lambda_{\mathcal{X}_{i,\eta}}\xrightarrow{\sim}\hat{\mathbb{Z}}^{p'}(1)$, the isomorphism $\lambda_{\Phi}$ is multiplication by $p^{n_{2}-n_{1}}$.
	\end{enumerate}
\end{lemma}

\begin{proof}
	Since $r\geq 5$, the group $\FinStepSolvQuo{\EtFundGrpGeom{\mathcal{X}_{1,\eta}}}{\abelian}$ is nontrivial.
	By Lemma~\ref{geomreco1}\ref{geompireco} and Lemma~\ref{geomreco2_inerreco}, the isomorphism $\Phi$ induces an isomorphism $\FinStepSolvQuoGeom{\Phi}{\abelian}\colon \FinStepSolvQuoGeom{\EtFundGrpPi{\mathcal{X}_{1,\eta}}}{\abelian}\xrightarrow{\sim}\FinStepSolvQuoGeom{\EtFundGrpPi{\mathcal{X}_{2,\eta}}}{\abelian}$ that preserves inertia groups.
	Hence $\Phi$ induces canonically an isomorphism $\lambda_{\Phi}\colon \Lambda_{\mathcal{X}_{1,\eta}}\xrightarrow{\sim}\Lambda_{\mathcal{X}_{2,\eta}}$.

	Let $s\in S^{\closed}$.
	By the same argument, $\Phi_{s}$ induces canonically an isomorphism $\lambda_{\Phi_{s}}\colon \Lambda_{\mathcal{X}_{1,s}}\xrightarrow{\sim}\Lambda_{\mathcal{X}_{2,s}}$.
	Let $(\overline{\mathcal{F}}_{s}(\Phi_{s}),\mathcal{F}_{s}(\Phi_{s}))$ be the image of $\Phi_{s}$ by Theorem~\ref{finGCweak}.
	For $i=1,2$, the specialization morphism induces a commutative diagram with exact rows
	\begin{equation*}
		\vcenter{
		\xymatrix{
		0\ar[r]
		&
		\Lambda_{\mathcal{X}_{i,\eta}}\ar[r]\ar[d]
		&
		\displaystyle\bigoplus_{x\in \mathcal{E}_{i,\eta}(\FieldSepClosure{\kappa(\eta)})}
		I_{x,\EtFundGrpGeom{\mathcal{X}_{i,\eta}}^{\abelian}}
		\ar[r]\ar[d]
		&
		\EtFundGrpGeom{\mathcal{X}_{i,\eta}}^{\abelian}
		\ar[r]\ar[d]
		&
		0
		\\
		0\ar[r]
		&
		\Lambda_{\mathcal{X}_{i,s}}\ar[r]
		&
		\displaystyle\bigoplus_{y\in \mathcal{E}_{i,s}(\FieldSepClosure{\kappa(s)})}
		I_{y,\EtFundGrpGeom{\mathcal{X}_{i,s}}^{\abelian}}
		\ar[r]
		&
		\EtFundGrpGeom{\mathcal{X}_{i,s}}^{\abelian}
		\ar[r]
		&
		0.
		}
		}
	\end{equation*}
	Since $\Phi_{s}$ is obtained from $\Phi$ by specialization, the right-hand two squares are commutative.
	Hence the left-hand square
	\begin{equation*}
		\vcenter{
		\xymatrix{
		\Lambda_{\mathcal{X}_{1,\eta}}
		\ar[r]^-{\lambda_{\Phi}}
		\ar[d]
		&
		\Lambda_{\mathcal{X}_{2,\eta}}
		\ar[d]
		\\
		\Lambda_{\mathcal{X}_{1,s}}
		\ar[r]^-{\lambda_{\Phi_{s}}}
		&
		\Lambda_{\mathcal{X}_{2,s}}
		}
		}
	\end{equation*}
	is also commutative.
	By the commutativity of the right-hand triangle in~\eqref{finGCweak_tridiag}, the isomorphism $\lambda_{\Phi_{s}}$ coincides with the morphism induced by $\overline{\mathcal{F}}_{s}(\Phi_{s})$.
	Under the canonical identifications $\Lambda_{\mathcal{X}_{i,s}}\xrightarrow{\sim}\hat{\mathbb{Z}}^{p'_{s}}(1)$, this is multiplication by $p_{s}^{\alpha(\Phi_{s})}$.
	Choose non-negative integers $n_{1,s}$ and $n_{2,s}$ such that $n_{2,s}-n_{1,s}\equiv \alpha(\Phi_{s})\mod [\kappa(s):\mathbb{F}_{p_{s}}]$.
	Therefore, the Frobenius twists of $\mathcal{F}_{s}(\Phi_{s})$ give a $\kappa(s)$-isomorphism $\mathcal{X}_{1,s}(n_{1,s})\xrightarrow{\sim}\mathcal{X}_{2,s}(n_{2,s})$.
	Hence $(\zeta_{1}|_{s})\circ \Frobenius_{\Spec(\kappa(s))}^{n_{1,s}}=(\zeta_{2}|_{s})\circ \Frobenius_{\Spec(\kappa(s))}^{n_{2,s}}$, so in particular $\zeta_{1}(s)=\zeta_{2}(s)$.
	Thus $\zeta_{1}$ and $\zeta_{2}$ coincide set-theoretically on $S^{\closed}$.

	\medskip

	\noindent
	\ref{lem:global-frobenius-discrepancy-char0}
	Let $s\in S^{\closed}$ be arbitrary.
	Since $S$ is dominant over $\Spec(\mathbb{Z})$, we may choose a closed point $s'\in S$ such that $p_{s}\neq p_{s'}$.
	The image of $\lambda_{\Phi}$ in $\Isom(\Lambda_{\mathcal{X}_{1,s}},\Lambda_{\mathcal{X}_{2,s}})$ is $p_{s}^{\alpha(\Phi_{s})}$, and the image of $\lambda_{\Phi}$ in $\Isom(\Lambda_{\mathcal{X}_{1,s'}},\Lambda_{\mathcal{X}_{2,s'}})$ is $p_{s'}^{\alpha(\Phi_{s'})}$.
	Hence the image of $\lambda_{\Phi}$ in $(\hat{\mathbb{Z}}^{\{p_{s},p_{s'}\}'})^{\times}$ lies in both the subgroup generated by $p_{s}$ and the subgroup generated by $p_{s'}$.
	By~\cite[Theorem~1]{MR0044570}, this intersection is trivial.
	Therefore $p_{s}^{\alpha(\Phi_{s})}=1$, hence $\alpha(\Phi_{s})=0$.
	Since $s$ was arbitrary, we obtain $\alpha(\Phi_{s})=0$ for every closed point $s\in S$.

	Let $\ell$ be a prime number.
	Choose a closed point $s\in S$ such that $p_{s}\neq \ell$.
	Hence the image of $\lambda_{\Phi}$ in $(\hat{\mathbb{Z}}^{p'_{s}})^{\times}$ is trivial, so its $\ell$-primary component is also trivial.
	Since $\ell$ was arbitrary, $\lambda_{\Phi}$ is the identity under the canonical identifications $\Lambda_{\mathcal{X}_{i,\eta}}\xrightarrow{\sim}\hat{\mathbb{Z}}(1)$.

	Since $S$ is of finite type over the Jacobson scheme $\Spec(\mathbb{Z})$, the set $S^{\closed}$ is dense in $S$.
	Since $M_{0,r}$ is separated over $\Spec(\mathbb{Z})$, the equalizer $(\zeta_{1},\zeta_{2})^{-1}(\Delta_{M_{0,r}/\Spec(\mathbb{Z})})$ is a closed subset of $S$ that contains $S^{\closed}$.
	Hence it is equal to $S$, so $\zeta_{1}=\zeta_{2}$.
	This completes the proof.

	\noindent
	\ref{lem:global-frobenius-discrepancy-charp}
	Assume next that $p>0$ and that $\dim(\overline{\zeta_{1}(S)})>0$.
	Therefore $p_{s}=p$ for every closed point $s\in S^{\closed}$.
	Again by the above discussion, the image of $\lambda_{\Phi}$ in $\Isom(\Lambda_{\mathcal{X}_{1,s}},\Lambda_{\mathcal{X}_{2,s}})\cong (\hat{\mathbb{Z}}^{p'})^{\times}$ is $p^{\alpha(\Phi_{s})}$.
	Since the homomorphism $\hat{\mathbb{Z}}\to (\hat{\mathbb{Z}}^{p'})^{\times}$, $\gamma\mapsto p^{\gamma}$, is injective, it follows that $\alpha(\Phi_{s})$ does not depend on $s$.
	We write its common value as $\alpha$.

	By the set-theoretic equality on $S^{\closed}$ and~\cite[Theorem~1.2.1]{MR2012864}, there exists a unique pair $(n_{1},n_{2})$ of non-negative integers with $n_{1}n_{2}=0$ such that
	\begin{equation*}
		\zeta_{1}\circ \Frobenius_{S}^{n_{1}}
		=
		\zeta_{2}\circ \Frobenius_{S}^{n_{2}}.
	\end{equation*}
	For every closed point $s\in S$, put $d_{s}\coloneq [\kappa(\zeta_{1}(s)):\mathbb{F}_{p}]=[\kappa(\zeta_{2}(s)):\mathbb{F}_{p}]$.
	Therefore the equalities above imply $n_{2}-n_{1}\equiv n_{2,s}-n_{1,s}\equiv \alpha\mod d_{s}$.
	Since $\overline{\zeta_{1}(S)}$ is irreducible and positive-dimensional,~\cite[Lemma~4.10]{MR4745885} implies that $n_{2}-n_{1}=\alpha$.
	Therefore $n_{2}-n_{1}=\alpha(\Phi_{s})$ for every closed point $s\in S$.
	Since the image of $\lambda_{\Phi}$ in $(\hat{\mathbb{Z}}^{p'})^{\times}$ is $p^{\alpha}$, it is multiplication by $p^{n_{2}-n_{1}}$.
	This completes the proof.
\end{proof}

\begin{theorem}\label{fingenemainthmchzero}
	Let $k$ be a field finitely generated over $\mathbb{Q}$.
	Let $i=1,2$.
	Let $(\Compactification{X_{i}},E_{i})$ be smooth curves of type $(g_{i},r_{i})$ over $k$, and set $X_{i}\coloneq \Compactification{X_{i}}\setminus E_{i}$.
	Assume that $g_{1}=0$ and that $r_{1}\geq 5$.
	Then for any $\AbsGalGrp{k}$-isomorphism
	\begin{equation*}
		\Phi\colon \FinStepSolvQuoGeom{\EtFundGrpPi{1}}{\metabelian}\xrightarrow{\sim}\FinStepSolvQuoGeom{\EtFundGrpPi{2}}{\metabelian},
	\end{equation*}
	there exists a unique $k$-isomorphism $u\colon X_{1}\xrightarrow{\sim}X_{2}$ such that the induced isomorphism
	\begin{equation*}
		\FinStepSolvQuoGeom{u}{\abelian}\colon \FinStepSolvQuoGeom{\EtFundGrpPi{1}}{\abelian}\xrightarrow{\sim}\FinStepSolvQuoGeom{\EtFundGrpPi{2}}{\abelian}
	\end{equation*}
	coincides with the isomorphism $\FinStepSolvQuoGeom{\Phi}{\abelian}$ induced by $\Phi$.
\end{theorem}

\begin{proof}
	By~\cite[Proposition~1.7(1)]{MR4745885} and Lemma~\ref{geomreco2}\ref{rreco}, we have $g_{2}=0$ and $r_{2}=r_{1}$.
	We write $r\coloneq r_{1}(=r_{2})$.
	We first assume that $E_{i}(\FieldAlgeClosure{k})=E_{i}(k)$ for $i=1,2$.
	After shrinking a regular model if necessary, we may choose a regular integral scheme $S$ of finite type over $\Spec(\mathbb{Z})$ with function field $k$ such that $(\Compactification{X}_{i},E_{i})$ extends to a smooth curve $(\Compactification{\mathcal{X}_{i}},\mathcal{E}_{i})$ of type $(0,r)$ over $S$.
	Write $\eta$ for the generic point of $S$.
	Let
	\begin{equation*}
		\zeta_{i}\colon S\to M_{0,r}
	\end{equation*}
	be the morphism classifying the ordered family $(\Compactification{\mathcal{X}_{i}},\mathcal{E}_{i})$.
	By Lemma~\ref{lem:global-frobenius-discrepancy}\ref{lem:global-frobenius-discrepancy-char0}, we have $\zeta_{1}=\zeta_{2}$.
	Since $M_{0,r}$ represents the moduli problem of smooth proper genus-$0$ curves equipped with $r$ disjoint ordered sections, there exists a unique $S$-isomorphism $(\Compactification{\mathcal{X}_{1}},\mathcal{E}_{1})\xrightarrow{\sim}(\Compactification{\mathcal{X}_{2}},\mathcal{E}_{2})$ that preserves the ordered sections.
	By taking the generic fiber, we obtain a $k$-isomorphism $u\colon X_{1}\xrightarrow{\sim}X_{2}$.
	Let
	\begin{equation*}
		\FinStepSolvQuoGeom{u}{\abelian}\colon \FinStepSolvQuoGeom{\EtFundGrpPi{X_{1}}}{\abelian}\xrightarrow{\sim}\FinStepSolvQuoGeom{\EtFundGrpPi{X_{2}}}{\abelian}
	\end{equation*}
	be the induced isomorphism.
	By construction, $\FinStepSolvQuoGeom{u}{\abelian}$ and $\FinStepSolvQuoGeom{\Phi}{\abelian}$ preserve inertia groups and induce the same bijection on the sets of cusps.
	Moreover, Lemma~\ref{lem:global-frobenius-discrepancy}\ref{lem:global-frobenius-discrepancy-char0} shows that $\FinStepSolvQuoGeom{\Phi}{\abelian}$ induces the identity on $\Lambda_{X_{1}}\xrightarrow{\sim}\Lambda_{X_{2}}$.
	The same is true for $\FinStepSolvQuoGeom{u}{\abelian}$, since $u$ preserves the ordered cusps.
	Since $X_{1}$ has genus $0$, the exact sequence~\eqref{wfseq} implies that $\FinStepSolvQuoGeom{u}{\abelian}=\FinStepSolvQuoGeom{\Phi}{\abelian}$.
	By \cite[Lemma~4.9]{MR4745885}, the natural map
	\begin{equation}\label{grwgraagr}
		\Isom_{k}(X_{1},X_{2})
		\to
		\Isom_{\AbsGalGrp{k}}(\FinStepSolvQuoGeom{\EtFundGrpPi{X_{1}}}{\abelian},\FinStepSolvQuoGeom{\EtFundGrpPi{X_{2}}}{\abelian})
	\end{equation}
	is injective.
	Thus such a $u$ is unique.
	This completes the proof in the case where $E_{i}(\FieldAlgeClosure{k})=E_{i}(k)$ for $i=1,2$.

	Next, we show the general case.
	Let $L$ be a finite Galois extension of $k$ such that $E_{i,L}(\FieldAlgeClosure{L})=E_{i,L}(L)$ for $i=1,2$.
	By the above argument, there exists an $L$-isomorphism $u_{L}\colon X_{1,L}\xrightarrow{\sim}X_{2,L}$ satisfying the condition in the statement.
	It is enough to show that $u_{L}$ is fixed by $\mathrm{Gal}(L/k)$.
	Let $\sigma\in \mathrm{Gal}(L/k)$.
	Since $\FinStepSolvQuoGeom{\Phi}{\abelian}$ is induced by an $\AbsGalGrp{k}$-isomorphism, the bijection on cusps determined by $\FinStepSolvQuoGeom{\Phi}{\abelian}$ is $\mathrm{Gal}(L/k)$-equivariant.
	Hence $\sigma\cdot u_{L}$ and $u_{L}$ induce the same bijection on the sets of cusps.
	Moreover, $\sigma\cdot u_{L}$ also preserves the ordered cusps, so it induces the identity on $\Lambda_{X_{1,L}}\xrightarrow{\sim}\Lambda_{X_{2,L}}$.
	Therefore the exact sequence~\eqref{wfseq} implies that $\sigma\cdot u_{L}=u_{L}$.
	By Galois descent, $u_{L}$ descends to a $k$-isomorphism $u\colon X_{1}\xrightarrow{\sim}X_{2}$.
	Since $\FinStepSolvQuoGeom{u_{L}}{\abelian}$ coincides with the isomorphism induced by $\Phi$, we obtain $\FinStepSolvQuoGeom{u}{\abelian}=\FinStepSolvQuoGeom{\Phi}{\abelian}$.
	The uniqueness of $u$ follows from the injectivity of \eqref{grwgraagr}.
	This completes the proof.
\end{proof}

We say that a smooth curve over a field of positive characteristic is \emph{isotrivial} if, after a finite extension of the base field, it descends to a smooth curve over a finite field.

\begin{theorem}\label{fingenemainthmchpositive}
	Let $k$ be a field finitely generated over $\mathbb{F}_{p}$.
	Let $i=1,2$.
	Let $(\Compactification{X_{i}},E_{i})$ be non-isotrivial smooth curves of type $(g_{i},r_{i})$ over $k$, and set $X_{i}\coloneq \Compactification{X_{i}}\setminus E_{i}$.
	Assume that $g_{1}=0$ and that $r_{1}\geq 5$.
	Then for any $\AbsGalGrp{k}$-isomorphism
	\begin{equation*}
		\Phi\colon \FinStepSolvQuoGeom{\EtFundGrpPi{1}}{\metabelian}\xrightarrow{\sim}\FinStepSolvQuoGeom{\EtFundGrpPi{2}}{\metabelian},
	\end{equation*}
	there exists a unique pair $(n_{1},n_{2})$ of non-negative integers with $n_{1}n_{2}=0$ and a unique $k$-isomorphism $u\colon X_{1}(n_{1})\xrightarrow{\sim}X_{2}(n_{2})$ such that, if we denote by $\iota_{i}\colon \FinStepSolvQuoGeom{\EtFundGrpPi{X_{i}(n_{i})}}{\abelian}\xrightarrow{\sim}\FinStepSolvQuoGeom{\EtFundGrpPi{i}}{\abelian}$ the natural isomorphism induced by the relative Frobenius morphism for each $i=1,2$, then the induced isomorphism
	\begin{equation*}
		\FinStepSolvQuoGeom{u}{\abelian}\colon \FinStepSolvQuoGeom{\EtFundGrpPi{X_{1}(n_{1})}}{\abelian}\xrightarrow{\sim}\FinStepSolvQuoGeom{\EtFundGrpPi{X_{2}(n_{2})}}{\abelian}
	\end{equation*}
	coincides with the composition $\iota_{2}^{-1}\circ\FinStepSolvQuoGeom{\Phi}{\abelian}\circ \iota_{1}$.
\end{theorem}

\begin{proof}
	By~\cite[Proposition~1.7(1)]{MR4745885} and Lemma~\ref{geomreco2}\ref{rreco}, we have $g_{2}=0$ and $r_{2}=r_{1}$.
	We write $r\coloneq r_{1}(=r_{2})$.
	We first assume that $E_{i}(\FieldSepClosure{k})=E_{i}(k)$ for $i=1,2$.
	After shrinking a regular model if necessary, we may choose a regular integral scheme $S$ of finite type over $\Spec(\mathbb{Z})$ with function field $k$ such that $(\Compactification{X}_{i},E_{i})$ extends to a smooth curve $(\Compactification{\mathcal{X}_{i}},\mathcal{E}_{i})$ of type $(0,r)$ over $S$.
	Write $\eta$ for the generic point of $S$.
	Let
	\begin{equation*}
		\zeta_{i}\colon S\to M_{0,r}
	\end{equation*}
	be the morphism classifying the ordered family $(\Compactification{\mathcal{X}_{i}},\mathcal{E}_{i})$.
	By Lemma~\ref{lem:global-frobenius-discrepancy}\ref{lem:global-frobenius-discrepancy-charp} and the non-isotriviality hypothesis, there exists a unique pair $(n_{1},n_{2})$ of non-negative integers with $n_{1}n_{2}=0$ such that
	\begin{equation*}
		\zeta_{1}\circ \Frobenius_{S}^{n_{1}}
		=
		\zeta_{2}\circ \Frobenius_{S}^{n_{2}}.
	\end{equation*}
	Since $M_{0,r}$ represents the moduli problem of smooth proper genus-$0$ curves equipped with $r$ disjoint ordered sections, there exists a unique $S$-isomorphism $(\Compactification{\mathcal{X}_{1}}(n_{1}),\mathcal{E}_{1}(n_{1}))\xrightarrow{\sim}(\Compactification{\mathcal{X}_{2}}(n_{2}),\mathcal{E}_{2}(n_{2}))$ that preserves the ordered sections.
	By taking the generic fiber, we obtain a $k$-isomorphism $u\colon X_{1}(n_{1})\xrightarrow{\sim}X_{2}(n_{2})$.
	Let
	\begin{equation*}
		\FinStepSolvQuoGeom{u}{\abelian}\colon \FinStepSolvQuoGeom{\EtFundGrpPi{X_{1}(n_{1})}}{\abelian}\xrightarrow{\sim}\FinStepSolvQuoGeom{\EtFundGrpPi{X_{2}(n_{2})}}{\abelian}
	\end{equation*}
	be the induced isomorphism.
	By construction, $\iota_{2}\circ\FinStepSolvQuoGeom{u}{\abelian}\circ \iota_{1}^{-1}$ and $\FinStepSolvQuoGeom{\Phi}{\abelian}$ preserve inertia groups and induce the same bijection on the sets of cusps.
	Since $u$ preserves the ordered cusps, it induces the identity on $\Lambda_{X_{1}(n_{1})}\xrightarrow{\sim}\Lambda_{X_{2}(n_{2})}$.
	Under the canonical identifications with $\hat{\mathbb{Z}}^{p'}(1)$, the Frobenius isomorphism $\Lambda_{X_{i}(n_{i})}\xrightarrow{\sim}\Lambda_{X_{i}}$ is multiplication by $p^{n_{i}}$.
	Hence $\iota_2\circ\FinStepSolvQuoGeom{u}{\abelian}\circ \iota_1^{-1}$ induces multiplication by $p^{n_{2}-n_{1}}$ on $\Lambda_{X_{1}}\xrightarrow{\sim}\Lambda_{X_{2}}$.
	By Lemma~\ref{lem:global-frobenius-discrepancy}\ref{lem:global-frobenius-discrepancy-charp} and the non-isotriviality hypothesis, the same is true for $\FinStepSolvQuoGeom{\Phi}{\abelian}$.
	Since $X_{1}$ has genus $0$, the exact sequence~\eqref{wfseq} implies that $\iota_2\circ\FinStepSolvQuoGeom{u}{\abelian}\circ \iota_1^{-1}=\FinStepSolvQuoGeom{\Phi}{\abelian}$.
	The uniqueness of $u$ follows from \cite[Lemma~4.9]{MR4745885}.
	This completes the proof in the case where $E_{i}(\FieldSepClosure{k})=E_{i}(k)$ for $i=1,2$.
	The proof of the general case is the same as that of Theorem~\ref{fingenemainthmchzero}.
	This completes the proof.
\end{proof}

\begin{corollary}\label{fingeneweak}
	Let $k$ be a field finitely generated over its prime field of characteristic $p\geq 0$.
	Let $i=1,2$.
	Let $X_{i}$ be smooth curves of type $(g_{i},r_{i})$ over $k$.
	Assume that $g_{1}=0$ and that $r_{1}\geq 5$.
	Then the following hold:
	\begin{enumerate}[(1)]
		\item \label{fingeneweakzero}
		      Assume moreover that $p=0$.
		      Then
		      \begin{equation*}
			      X_{1}
			      \cong_{k}
			      X_{2}
			      \iff
			      \FinStepSolvQuoGeom{\EtFundGrpPi{X_{1}}}{\metabelian}
			      \cong_{\AbsGalGrp{k}}
			      \FinStepSolvQuoGeom{\EtFundGrpPi{X_{2}}}{\metabelian}.
		      \end{equation*}

		\item \label{fingeneweakpositive}
		      Assume moreover that $p>0$ and that $X_{1}$ and $X_{2}$ are non-isotrivial.
		      Then
		      \begin{equation*}
			      \FinStepSolvQuoGeom{\EtFundGrpPi{X_{1}}}{\metabelian}
			      \cong_{\AbsGalGrp{k}}
			      \FinStepSolvQuoGeom{\EtFundGrpPi{X_{2}}}{\metabelian},
		      \end{equation*}
		      if and only if there exists a pair $(n_{1},n_{2})$ of non-negative integers such that
		      \begin{equation*}
			      X_{1}(n_{1})\cong_{k}X_{2}(n_{2}).
		      \end{equation*}
	\end{enumerate}
\end{corollary}

\begin{proof}
	The assertions follow from Theorems~\ref{fingenemainthmchzero} and~\ref{fingenemainthmchpositive}.
\end{proof}

Note that, when $p>0$, smooth curves of type $(0,3)$ are isotrivial.

\begin{proposition}\label{fingeneweak03}
	The same statement as in Corollary~\ref{fingeneweak}\ref{fingeneweakzero} holds when $(g_{1},r_{1})=(0,3)$.
\end{proposition}

\begin{proof}
	Let $\Phi\in \Isom_{\AbsGalGrp{k}}(\FinStepSolvQuoGeom{\EtFundGrpPi{1}}{\metabelian},\FinStepSolvQuoGeom{\EtFundGrpPi{2}}{\metabelian})$.
	Hence $g_{1}=g_{2}=0$ and $r_{1}=r_{2}=3$ by~\cite[Proposition~1.7(1)]{MR4745885} and Lemma~\ref{geomreco2}\ref{rreco}.
	Lemma~\ref{geomreco2_inerreco} implies that $\Phi$ induces an isomorphism $\mathrm{Iner}(\FinStepSolvQuo{\EtFundGrpGeom{X_{1}}}{\abelian})\xrightarrow{\sim}\mathrm{Iner}(\FinStepSolvQuo{\EtFundGrpGeom{X_{2}}}{\abelian})$ that is compatible with the actions of $\AbsGalGrp{k}$.
	We know that $E_{i}(\FieldAlgeClosure{k})\to \mathrm{Iner}(\FinStepSolvQuo{\EtFundGrpGeom{i}}{\abelian})$ is a $\AbsGalGrp{k}$-isomorphism (see~\cite[Lemma~1.11]{MR4745885}).
	Hence $\Phi$ induces a bijection $E_{1}(\FieldAlgeClosure{k})\xrightarrow{\sim}E_{2}(\FieldAlgeClosure{k})$ that is compatible with the actions of $\AbsGalGrp{k}$.
	Since $r=\CardinalityOfSet{E_{i}(\FieldAlgeClosure{k})}=3$, there exists a unique $\FieldAlgeClosure{k}$-isomorphism $\tilde{\phi}\colon (\Compactification{X}_{1})_{\FieldAlgeClosure{k}}\xrightarrow{\sim}(\Compactification{X}_{2})_{\FieldAlgeClosure{k}}$ that induces the same bijection on cusps.
	Hence $\tilde{\phi}$ is compatible with the actions of $\AbsGalGrp{k}$.
	Therefore, by Galois descent, $\tilde{\phi}$ descends to an isomorphism $\Compactification{X}_{1}\xrightarrow{\sim}\Compactification{X}_{2}$ over $k$, hence induces an isomorphism $X_{1}\xrightarrow{\sim}X_{2}$ over $k$.
	Thus, the assertions in Corollary~\ref{fingeneweak}\ref{fingeneweakzero} hold.
\end{proof}

\begin{corollary}
	Let $k$ be a field finitely generated over its prime field of characteristic $p\geq 0$.
	Let $i=1,2$.
	Let $X_{i}$ be smooth curves of type $(g_{i},r_{i})$ over $k$.
	Assume that $X_{1}$ and $X_{2}$ are non-isotrivial if $p>0$.
	Assume that $g_{1}=0$ and that $r_{1}\geq 5$.
	Let
	\begin{equation*}
		\Phi\colon \FinStepSolvQuoGeom{\EtFundGrpPi{1}}{\metabelian}\xrightarrow{\sim}\FinStepSolvQuoGeom{\EtFundGrpPi{2}}{\metabelian}
	\end{equation*}
	be an $\AbsGalGrp{k}$-isomorphism.
	Then the induced isomorphism $\FinStepSolvQuoGeom{\Phi}{\abelian}$ preserves decomposition groups.
\end{corollary}

\begin{proof}
	Since $\FinStepSolvQuoGeom{\Phi}{\abelian}$ is geometric by Theorems~\ref{fingenemainthmchzero} and~\ref{fingenemainthmchpositive}, the assertion is immediate.
\end{proof}

\section*{Acknowledgments}
The author is sincerely grateful to Prof. Akio~Tamagawa for his generous guidance and helpful suggestions throughout this work.
The author also thanks Prof. Yuichiro~Taguchi, Takahiro Murotani, and Shun~Ishii for their advice.
This work was supported by the Japan Society for the Promotion of Science (JSPS) KAKENHI Grant Number 23KJ0881.

\printbibliography

\end{document}